\documentclass{article}
\usepackage{graphicx}
\usepackage{amsmath}
\usepackage{amssymb}
\usepackage{amsthm}
\usepackage{amsfonts}
\usepackage{enumerate}
\usepackage{epstopdf}
\usepackage[noend]{algorithmic}
\usepackage{algorithm,caption}

\usepackage{array}
\usepackage{subfig}
\usepackage[usenames,dvipsnames,svgnames]{xcolor}
\usepackage{tikz}
\tikzset{
  font={\fontsize{8pt}{12}\selectfont}}
\usetikzlibrary{decorations.markings}
\usepackage{tikzscale}
\usepackage{pgfplots}
\usepgfplotslibrary{patchplots,colorbrewer,groupplots}

\usepackage{multirow}
\usepackage{lipsum}
\usepackage[margin=3cm]{geometry}
\usepackage{hyperref}
\usepackage{cleveref}

\title{Function approximation on arbitrary domains using Fourier extension frames}

\author{Roel Matthysen, Daan Huybrechs}

\newcommand{\os}{\gamma}
\newcommand{\e}{e}
\renewcommand{\imath}{\mathrm{i}}

\newcommand{\bszero}{\mathbf{0}}
\newcommand{\bsn}{\boldsymbol{n}}
\newcommand{\bsxi}{\boldsymbol{\xi}}
\newcommand{\bsx}{\boldsymbol{x}}
\newcommand{\bsy}{\boldsymbol{y}}
\newcommand{\bsl}{\mathbf{l}}
\newcommand{\bsk}{\mathbf{k}}
\newcommand{\bsa}{\mathbf{a}}
\newcommand{\bsc}{\mathbf{c}}

\newcommand{\matlab}{\begin{sc} matlab \end{sc}}

\newtheorem{thm}{Theorem}

\newtheorem{mydef}{Definition}
\newtheorem{remark}{Remark}

\DeclareMathOperator*{\argmin}{arg\,min} 
\DeclareMathOperator*{\argmax}{arg\,max} 

\definecolor{mygreen}{RGB}{28,172,0} 

\newlength{\figurewidth}
\newlength{\figureheight}

\DeclareMathOperator{\spn}{span}

\usepackage[final]{listings}
\newfloat{MATLAB code}{h}{}
\date{\today} 

\begin{document}

\maketitle

\begin{abstract}
Fourier extension is an approximation scheme in which a function on an arbitary
bounded domain is approximated using a classical Fourier series on a bounding
box. On the smaller domain the Fourier series exhibits redundancy, and it has
the mathematical structure of a frame rather than a basis. It is not trivial to
construct approximations in this frame using function evaluations in points that
belong to the domain only, but one way to do so is through a discrete least
squares approximation. The corresponding system is extremely ill-conditioned,
due to the redundancy in the frame, yet its solution via a regularized SVD is
known to be accurate to very high (and nearly spectral) precision. Still, this
computation requires ${\mathcal O}(N^3)$ operations. In this paper we describe
an algorithm to compute such Fourier extension frame approximations in only
${\mathcal O}(N^2 \log^2 N)$ operations for general 2D domains. The cost
improves to ${\mathcal O}(N \log^2N)$ operations for simpler tensor-product
domains. The algorithm exploits a phenomenon called the plunge region in the
analysis of time-frequency localization operators, which manifests itself here
as a sudden drop in the singular values of the least squares matrix. It is known
that the size of the plunge region scales like ${\mathcal O}(\log N)$ in one
dimensional problems. In this paper we show that for most 2D domains in the
fully discrete case the plunge region scales like ${\mathcal O}(N \log N)$,
proving a discrete equivalent of a result that was conjectured by Widom for a
related continuous problem. The complexity estimate depends on the Minkowski or
box-counting dimension of the domain boundary, and as such it is larger than ${\mathcal O}(N \log N)$ for domains with fractal shape.



\end{abstract}

\section{Introduction}

The nature of the discrete representation of a continuous function is an important choice in many applications. A
representation needs to be efficiently constructed as well as easy to use. The number of degrees of freedom in
the representation influences both of these considerations. We focus on methods that offer a
high rate of convergence for increasing degrees of freedom, specifically converging at least superalgebraically.

In two dimensions and higher, the domain of the function is an important
complication in the approximation process. When the domain has some structure,
such as a rectangle, it is possible to modify existing one-dimensional spectral
or high order methods to fit the domain. When the domain shape is arbitrary,
in many cases one resorts to a polygonal mesh with basis functions that have low orders of smoothness. An alternative that does offer spectral accuracy is the use of Radial Basis
Functions, suitable for unstructured data as long as shape parameters are well chosen \cite{Fornberg2008a}.

The representation used in this paper is based on orthogonal
basis functions on a bounding box that encloses the domain. The objective is to
obtain an expansion in this basis that resembles a given function as
closely as possible on the domain. An example of a
possible resulting approximation is shown in \cref{fig:example}, which contains
all the main elements of this procedure. In \cref{fig:example1} the
function $f(x,y)=\cos(20x^2-15y^2)$ is shown, where the domain takes on the
shape of the country Belgium. This function
is approximated using Fourier basis functions on a bounding
rectangle. Therefore the expansion in
\cref{fig:example1b} is periodic. Such an approximation is straightforward to compute when the function to be approximated can be evaluated throughout the bounding box. Yet, the approximation problem itself becomes challenging when one is restricted to function samples in the irregular domain only, and that is the setting we pursue.

\setlength{\figurewidth}{8cm}
\setlength{\figureheight}{4cm}

\begin{figure}[hbtp]
  \centering
  \subfloat[Data points]{\includegraphics[scale=0.15]{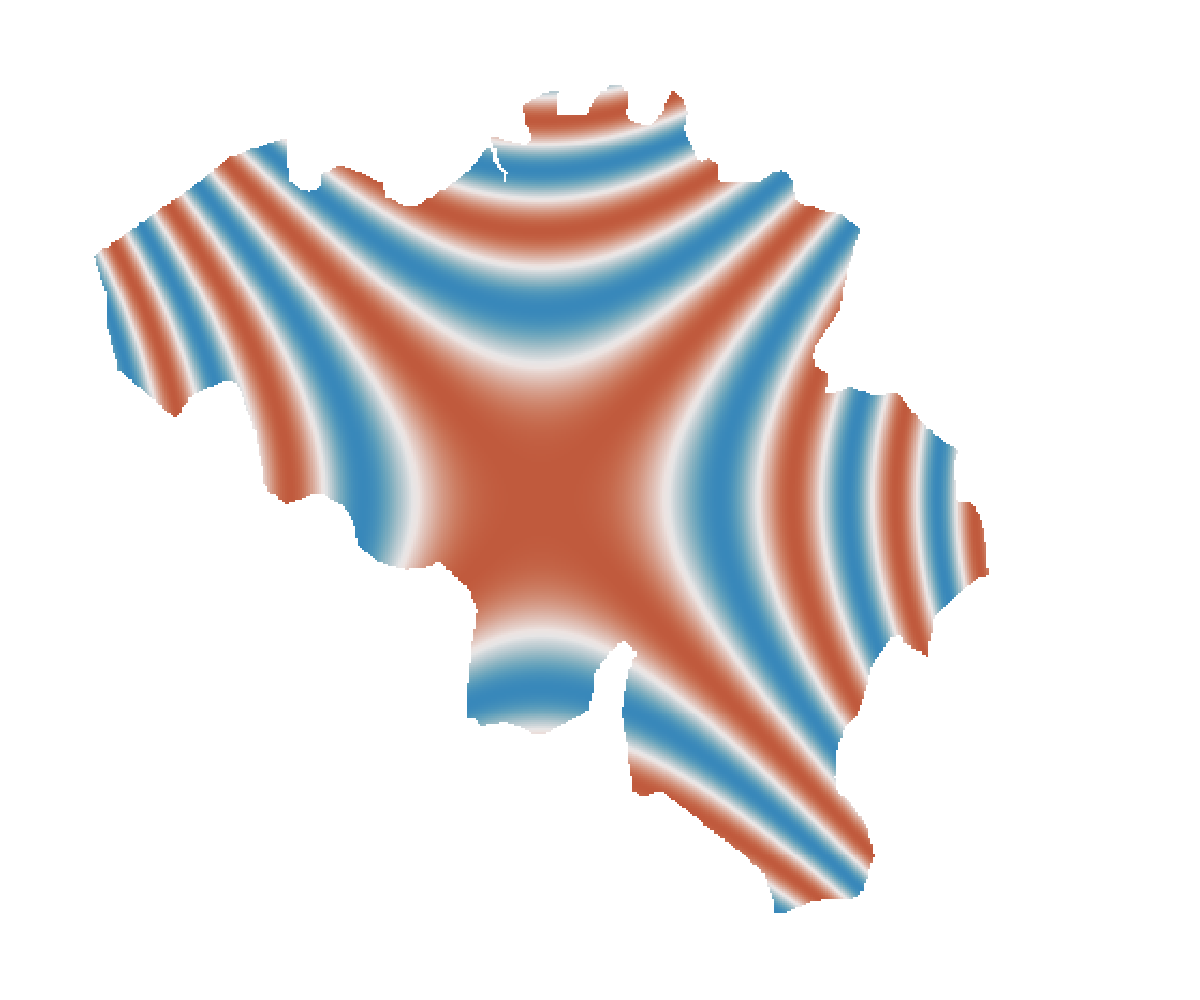}\label{fig:example1}}
  \subfloat[Fourier Series on bounding box]{\includegraphics[scale=0.15]{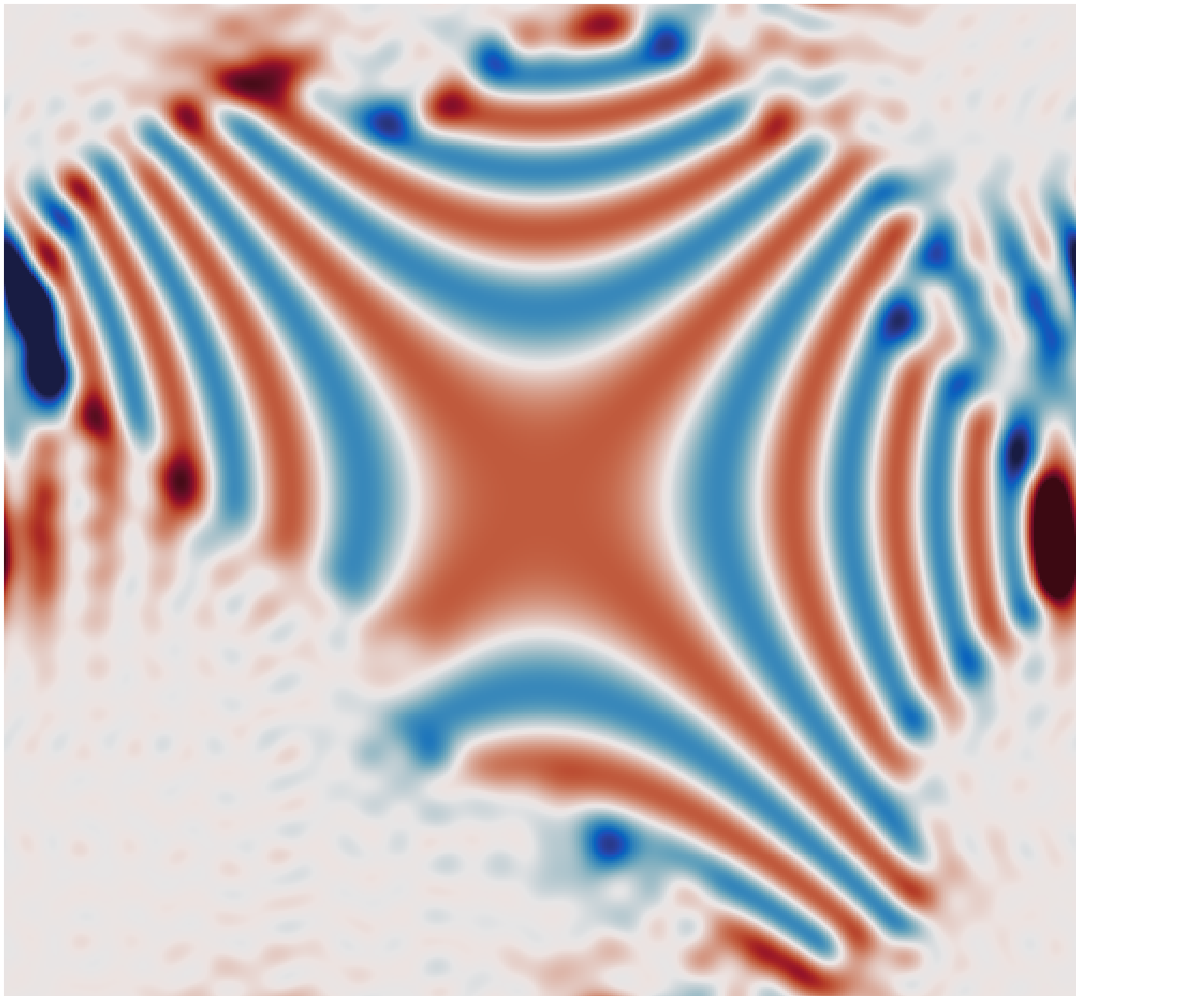}\label{fig:example1b}}
  \caption{Approximation of $f(x,y)=\cos(20x^2-15y^2)$ on a Belgium-shaped
    domain, using a Fourier series on a bounding box. }
  \label{fig:example}
\end{figure}

This method for one-dimensional Fourier bases is known as the Fourier extension (FE) or Fourier continuation (FC) technique \cite{Boyd2002, Boyd2005, Bruno2003, Bruno2007, Huybrechs2010}. In this technique the
approximation is computed by performing a least squares approximation on an equispaced grid of
collocation points. This allows the use of the FFT, which facilitates fast
algorithms. The main problem is the ill-conditioning of the collocation matrix. 

The FE method is closely related to embedded or fictitious domain methods for solving certain partial differential equations using
Fourier basis functions,
the main difference being the approximation in the extension region. In embedded
domain methods the function is explicitly extended outside the domain of
interest, e.\ g.\ through convolution with Gaussian kernels
\cite{Bueno-Orovio2006a} or using polynomial corrections \cite{Platte2009}. In the
Fourier extension technique, the approximation in the extension region is determined implicitly through solving a least
squares problem.  

Convergence properties of FE for function approximation were first described in \cite{Boyd2002} and \cite{Bruno2003}. It was shown that a Fourier basis approximation can converge
exponentially to $f$ inside the domain of interest. The approximation problem can be solved using discrete least squares and a truncated singular value decomposition. Bruno and Pohlmann used the FE
technique in higher dimensions to obtain smooth and periodic extensions around boundaries of complicated surfaces. This was the basis for the very efficient
FC-Gram method for approximations and solving differential equations, using 1D Fourier extensions combined with an ADI approach \cite{Bruno2010, Lyon2010, Albin2011}. Analysis of Fourier extensions continued in the context of frames, with precise convergence rates and error bounds now available \cite{Huybrechs2010,Adcock2011,Adcock2012,Adcock2016}. A major conclusion in the latter references is that, in spite of potentially extreme ill-conditioning, the regularized SVD solution of the discrete problem is numerically stable once sufficiently many degrees of freedom are used. Furthermore, a discrete least squares is more accurate than projection methods using the Gram matrix. This is true not only for Fourier extension, but more generally for numerical approximations in frames \cite{Adcock2016}.

Although the FE technique is very flexible when it comes to domain
choice and yields fast converging approximations, computing the expansion is not
efficient. The singular value decomposition is cubic in the number of degrees of freedom, making it very slow for even mildly oscillatory functions. 
In \cite{Lyon2011} and \cite{Matthysen2015}, two distinct $O(N\log^2 N)$ algorithms
for the 1D approximation scheme were introduced, where $N$ is the number of
degrees of freedom. 

The $O(N\log^2 N)$ complexity of the latter reference \cite{Matthysen2015} originates in the connection of this problem to Prolate Spheroidal Wave theory \cite{Slepian1978a,Landau1961,Pollack1961}. It allows for an interpretation of the singular values and vectors of the problem, in terms of functions that are maximally concentrated simultaneously in time and frequency. This is a topic in signal processing mostly, but in our context it shows that the singular values can be separated into three clusters: one cluster of values is exponentially close to $1$, a second one is exponentially close to $0$, and a third transitional set in between contains values exponentially dropping from $1$ to machine precision. This very particular distribution of singular values is illustrated in \cref{fig:intervals} further on and is seen for domains of any shape. The latter transitional set is called the \emph{plunge region}. Crucially, the size of this plunge region has smaller complexity than the other clusters: it is known for one-dimensional problems that it grows only as $O(\log{N})$ \cite{Wilson1987}. The fast algorithm of \cite{Matthysen2015} is based on a projection of the problem onto a smaller one, with dimensions governed by the size of the plunge region, which is solved with a direct method. It is followed by a post-processing step that amounts to a single FFT. Details are recalled in \cref{sec:problem} of this paper. In spite of its mathematical intricacies, the algorithm itself is short and simple, and a full implementation for the 1D case was included in the appendix of \cite{Matthysen2015}.

It is clear that the cost of the algorithm is directly influenced by the size of the plunge region. Unfortunately, the plunge region for 2D domains is relatively larger than it is in 1D. For tensor-product domains, the tensor product structure can be readily exploited to generalize the algorithm of \cite{Matthysen2015} with near-optimal complexity. For other domains, somewhat surprisingly, the plunge--projection method still applies virtually without modification. However, the size of the plunge region has not received as much study as in the one-dimensional case. In this paper we establish the size of the plunge region, and, hence, the computational cost of the algorithm. Loosely speaking, the plunge region scales with the size of the boundary of the domain, rather than with the domain itself. As such, it is a lower-dimensional phenomenon, which we confirm in this paper for the 2D case\footnote{Julia code for this algorithm is incorporated in
  the FrameFun package, at \url{http://github.com/daanhb/FrameFun.jl}}.

\subsection{Overview and main results of the paper}

In \cref{sec:problem} of this paper we describe the application of the 1D algorithm to
the multidimensional case, and we state the conditions that are necessary for a speedup compared to the full direct solver with cubic cost. This entails a proper choice of sample points in time and frequency.

In \cref{sec:spectrum} we show a bound on the singular value distribution of the collocation matrix. As indicated above, this corresponds to an estimate of the size of the plunge region for the discrete Fourier transform, for two-dimensional domains in the time domain and a rectangular region in the frequency domain. This bound is a direct generalization of the discrete time-frequency localization results of Wilson in \cite{Wilson1987}, and are a
discrete analog of recent results by Sobolev \cite{Sobolev2010,Sobolev2015} on a related conjecture by Widom \cite{Widom1982}. Most of our results, but not all, are formulated for a general dimension $D$.

\Cref{sec:experiments} contains numerical results for approximations on a variety of domains in a Fourier basis, demonstrating the flexibility, accuracy and speed of our approach.

The contribution of this paper is twofold: we show the accuracy and reliability of Fourier frame approximations in higher dimensions, while providing
a deeper understanding of the spectral properties of the problem through explicit bounds. This leads to an algorithm that is asymptotically faster than a direct solver on a rank-deficient rectangular least squares problem.



\newcommand{\domt}{\Omega}
\newcommand{\domtp}{P_\Omega}
\newcommand{\idomtp}{I_\Omega}
\newcommand{\tdom}{S_\Omega}
\newcommand{\ndomtp}{N_\Omega}
\newcommand{\domb}{\delta\Omega}
\newcommand{\dombp}{P_{\delta\Omega}}
\newcommand{\ndombp}{N_{\delta\Omega}}
\newcommand{\bboxt}{R}
\newcommand{\bboxtp}{P_R}
\newcommand{\ibboxtp}{I_R}
\newcommand{\nbboxtp}{N_R}
\newcommand{\singletp}{n_R}
\newcommand{\domf}{\Lambda}
\newcommand{\domfp}{P_\Lambda}
\newcommand{\ndomfp}{N_\Lambda}
\newcommand{\bboxf}{\tilde{R}}
\newcommand{\bboxfp}{P_{\tilde{R}}}
\newcommand{\nbboxfp}{N_{\tilde{R}}}
\newcommand{\singlefp}{n_\Lambda}
\newcommand{\Neps}{\eta(\epsilon,\nbboxtp)}
\newcommand{\Ntwo}{N^D}
\newcommand{\TBT}{T_{\domt} B_\Lambda T_{\domt}}
\newcommand{\BTTB}{B_\Lambda T_{\domt}B_\Lambda}
\newcommand{\fframe}{\mathcal{F}}
\newcommand{\gframe}{\mathcal{G}}
\newcommand{\dimi}{d}
\newcommand{\Dim}{D}
\newcommand{\dof}{N^D}

\section{Problem formulation}
\label{sec:problem}

\subsection{Approximation in a Fourier frame}

Informally, the problem formulation is as follows: given a function $f$ and a domain $\domt \subset \mathbb{R}^D$, find a Fourier series $F$ that minimizes a suitable norm $\Vert F-f \Vert_X$ on $\domt$, using only information about $f$ on $\domt$.

Let us elaborate and be precise. Without loss of generality, assume that $\domt
\subset \bboxt = [0,1]^D$, so that we can use a tensor-product of the standard
Fourier series on $[0,1]$. In order to avoid any periodicity requirements on $f$, we assume further that $\domt$ lies fully in the interior of the box $\bboxt$. In the following, we will consistently use the symbols $\domt$ for the time domain, and $\domf$ in the frequency domain. For an index set $\domfp$ with $\ndomfp$ frequencies, we denote the basis functions and the function space they span by
\begin{align}
  \label{eq:16}
\phi_{\bsl}(\bsx)&=\e^{\imath(\bsx\cdot\bsl)2\pi},\\
    \gframe_{\ndomfp}&=\spn\{\phi_{\bsl}\}_{\bsl\in\domfp}.
\end{align}
Here, $\bsx=(\bsx_1,\dots,\bsx_D)$ is a $D$-dimensional point and
$\bsl=(\bsl_1,\dots,\bsl_D)$ is a $D$-dimensional integer index. For simplicity, we assume an equal number of degrees of freedom $\singlefp$ per dimension, hence $\ndomfp=\singlefp^D$ and $\lfloor \singlefp/2\rfloor <\bsl_i<\lceil \singlefp/2 \rceil$. This restriction could be lifted at the cost of minor complications further on, and we do not have this restriction in our implementation.

The approximation problem in this space is stated as
\begin{equation}
  \label{eq:17}
  F_{\ndomfp}(f)=\min_{ g \in \gframe_{\ndomfp} }  \Vert f-g \Vert_X.
\end{equation}
Note here that the function set $\gframe$ restricted to $\domt$ is not
a basis for $\mathcal{L}^2(\Omega)$, but it is a frame in the sense of Duffin and Schaeffer \cite{Duffin1952}. For
a recent overview of numerical frame approximations, see \cite{Adcock2016}.

Any $g\in\gframe$ is uniquely described by a set of coefficients
$\bsa\in\mathbb{C}^{\singlefp\times\dots\times\singlefp}$. In the remainder we
will often assume an implicit linearization
$\bsa\in\mathbb{C}^{\ndomfp}$. These coefficients will be the result of the approximation algorithm, so we look for
\begin{equation}
  \label{eq:15}
  \bsa= \argmin_{\bsc\in\mathbb{C}^{\ndomfp}}\Vert f-\sum_{\bsl\in \domfp} \bsc_{\bsl}\phi_{\bsl} \Vert_X.
\end{equation}

It remains to define the norm $\Vert \cdot \Vert_X$. The choice of the $\mathcal{L}^2$ norm over $\domt$ is fairly natural, and this choice leads to the so-called
\emph{Continuous Frame approximation}. The minimizer of \cref{eq:15} is found by constructing the Gram matrix
\begin{equation}
  \label{eq:35}
  A^{G}_{\bsk,\bsl}= \langle \phi_\bsk,\phi_\bsl \rangle_\domt,
\end{equation}
and solving the system
\begin{equation}
  \label{eq:continuous_system}
  A^{G}\bsa=b,\qquad b_{\bsl}= \langle f, \phi_\bsl \rangle_\domt.
\end{equation}
However, the computational cost associated with evaluating the integrals in the right hand side of \eqref{eq:continuous_system} is considerable, since the integrals are over $\domt$ and not over the full box $\bboxt$. This precludes the use of the FFT and one would have to resort to some type of quadrature on $\domt$.

\begin{figure}[hbtp]
\centering
\subfloat[time domain]{\begin{tikzpicture}
    \foreach \x in {-2,-1.6,...,2}{
      \foreach \y in {-2,-1.6,...,2}{
        \pgfmathparse{(1.6+0.5*cos(5*(atan(\y/(\x+0.005))-90+90*sign(\x+0.005))))^2>(\x*\x+\y*\y) ? int(1) : int(0)}
        \ifnum\pgfmathresult=1
           \fill (\x,\y) circle(0.05cm);
        \else
           \draw (\x,\y) circle(0.05cm);
        \fi
}
}
\draw[domain=0:360,scale=1,samples=500] plot (\x:{1.6+0.5*cos(5*\x)});
\fill[color=white] (0,0) circle(0.3cm);
\node (pomega) at (0,0) {$P_\Omega$};
\fill[color=white] (1.6,1.6) circle(0.3cm);
\node (pr) at (1.6,1.6) {$P_R$};
\draw (-2.2,-2.2) rectangle (2.2,2.2);
  \end{tikzpicture}}\hspace{40pt}
\subfloat[frequency domain]{  \begin{tikzpicture}
    \foreach \x in {-2,-1.6,...,2}{
      \foreach \y in {-2,-1.6,...,2}{
        \pgfmathparse{max(abs(\x),abs(\y))<1.5 ? int(1) : int(0)}
        \ifnum\pgfmathresult=1
           \fill (\x,\y) circle(0.05cm);
        \else
           \draw (\x,\y) circle(0.05cm);
        \fi
}
}
\fill[color=white] (0,0) circle(0.3cm);
\node (pomega) at (0,0) {$P_\Lambda$};
\fill[color=white] (1.6,1.6) circle(0.3cm);
\node (pr) at (1.6,1.6) {$P_{\tilde{R}}$};
\draw (-2.2,-2.2) rectangle (2.2,2.2);
\draw (-1.4,-1.4) rectangle (1.4,1.4);
  \end{tikzpicture}}
    \caption{The spatial domain $\domt$ encompassing the sample set $\domtp$, and the frequency domain $\domf$ encompassing the discrete frequencies $\domfp$. There is a fast FFT transform between the encompassing sets $\bboxtp$ and $\bboxfp$.}
\label{fig:samplepoints}
\end{figure}
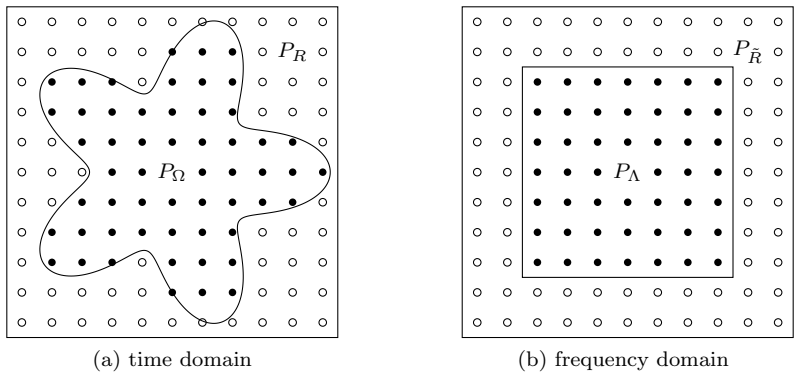

Instead, in this paper we will focus on the \emph{Discrete Frame approximation}. The corresponding norm is a discrete summation over a set of collocation points. For Fourier
approximations, we choose a set of equispaced points $\bboxtp$ on $\bboxt$
(recall that $\bboxt = [0,1]^D$) with $\singletp$ points per dimension, and
restrict those to $\domt$. In summary
\begin{equation}
  \label{eq:19}
\bboxtp = \left\{\left. \left(\frac{k_1}{\singletp},\dots,\frac{k_D}{\singletp}
    \right) \right| \forall i: 0\leq k_i< \singletp \right\}, \qquad\domtp=\bboxtp \cap \domt.
\end{equation}

There are efficient transformations using the FFT between the set $\bboxtp$ of
$\nbboxtp = \singletp^D$ points in the time domain and the set $\bboxfp$ in the
frequency domain. If we choose $\singletp \geq \singlefp$, then $\bboxfp$
encompasses $\domfp$. The sampling sets thus defined are shown in
\cref{fig:samplepoints}. Though other choices can be made, this choice is such
that we can efficiently evaluate a Fourier series using the index set $\domfp$
in all the points of $\domtp$: we extend the coefficients with zeros from
$\domfp$ to $\bboxfp$, followed by an FFT transform from $\bboxfp$ to $\bboxtp$,
followed by a restriction of the values to those points in $\domtp$. 

Throughout this paper we assume a fixed oversampling rate, meaning
$\ndomfp/\nbboxfp$ is constant. We refer to \cite{Adcock2013} for a study on the
interplay of oversampling rate and choice of bounding box in one dimension.

The minimization (\ref{eq:15}) can be reformulated as a discrete least squares problem
\begin{equation}
  \label{eq:17b}
  F_{\ndomfp}(f)=\argmin_{g\in\gframe_{\ndomfp}}  \sum_{
\bsx\in\domtp}(f(\bsx)-g(\bsx))^2.
\end{equation}
Assuming a linear indexing $\bsx_k$ of $\domtp$ from $1$ to $\ndomtp$ and $\phi_j$ of $\domfp$ from $1$ to
$\ndomfp$, it can be written as a least squares matrix problem
\begin{equation}
  \label{eq:lsq}
  A\bsa = b, \quad A\in\mathbb{C}^{\ndomtp\times\ndomfp}, \quad b\in\mathbb{C}^{\ndomtp}
\end{equation}
where
\begin{equation}
\label{eq:collocationdefinition}
A_{kj}=\frac{1}{\sqrt{\nbboxtp}}\phi_{j}(\bsx_k),\qquad b_{k}=f(\bsx_k).
\end{equation}
The scaling of the basisfunctions is such that $A$ is precisely a subblock of a
multidimensional unitary DFT matrix, which is of importance to \cref{alg:dpss2}.
This subblock property is a consequence of our choice of discrete grids, and it results in a fast matrix-vector product using the procedure described above: extension in frequency domain, discrete Fourier transform, and restriction in the time domain. Indeed, note that in this discrete setting the action of the matrix $A$ corresponds to evaluating a length $\ndomfp$ Fourier series in the points of $\domtp$.

Note that there is also a fast matrix-vector product for $A'$ and it corresponds to the opposite sequence of operations: extension from $\domtp$ to $\bboxtp$ by zeros, fast transform to $\bboxfp$, followed by restriction in the frequency domain from $\bboxfp$ to $\domfp$. Yet, it is clear that the solution to $Ax=b$ is not simply given by $x=A'b$: the latter is accurate only when the extension of the function $f$ by zero is well approximated by a Fourier series, which in general it is not. Extension by zero in the time domain introduces a discontinuity. On the other hand, extending a Fourier series with additional zero coefficients does not affect the function it describes.

\begin{remark}
Though the examples shown use Fourier bases exclusively, a Chebyshev collocation matrix in Chebyshev points is entirely analogous. It consists of a scaled subblock of a multidimensional DCT matrix. Therefore, most of the arguments, although not made explicit in this paper, apply to this case as well. This is of course due to the close connection between Chebyshev polynomials and trigonometric polynomials.
\end{remark}

\subsection{An approximation algorithm by projection onto the plunge region}

\label{sec:algorithm}

\Cref{eq:lsq} corresponds to a dense, rectangular linear system that is rank-deficient: the condition number of $A$ is exponentially large. Yet, high accuracy can be achieved with direct solvers such as a pivoted QR decomposition
(\matlab's backslash). Iterative algorithms, on the other hand, are much less suitable, in spite of a fast matrix-vector product being available, due to the ill-conditioning.

\setlength{\figurewidth}{8cm}
\setlength{\figureheight}{4cm}
\begin{figure}[hbtp]
\centering
 \begin{tikzpicture}
      \node (A) at (0,0){
%
%
\begin{tikzpicture}

\begin{axis}[%
width=\figurewidth,
height=\figureheight,
scale only axis,
xmin=101,
xmax=401,
xtick={200.5},
xticklabels={$\frac{\ndomtp\ndomfp}{\nbboxtp}$},
xlabel={i},
ymode=log,
ymin=1e-18,
ymax=10,
ytick={1e-16, 1e-08,     1},
yminorticks=true,
ylabel={$\sigma{}_\text{i}$}
]
\addplot [
color=blue,
solid,
forget plot
]
table[row sep=crcr]{
1 1\\
2 1\\
3 1\\
4 1\\
5 1\\
6 1\\
7 1\\
8 1\\
9 1\\
10 1\\
11 1\\
12 1\\
13 1\\
14 1\\
15 1\\
16 1\\
17 1\\
18 1\\
19 1\\
20 1\\
21 1\\
22 1\\
23 1\\
24 1\\
25 1\\
26 1\\
27 1\\
28 1\\
29 1\\
30 1\\
31 1\\
32 1\\
33 1\\
34 1\\
35 1\\
36 1\\
37 1\\
38 1\\
39 1\\
40 1\\
41 1\\
42 1\\
43 1\\
44 1\\
45 1\\
46 1\\
47 1\\
48 1\\
49 1\\
50 1\\
51 1\\
52 1\\
53 1\\
54 1\\
55 1\\
56 1\\
57 1\\
58 1\\
59 1\\
60 1\\
61 1\\
62 1\\
63 1\\
64 1\\
65 1\\
66 1\\
67 1\\
68 1\\
69 1\\
70 1\\
71 1\\
72 1\\
73 1\\
74 1\\
75 1\\
76 1\\
77 1\\
78 1\\
79 1\\
80 1\\
81 1\\
82 1\\
83 1\\
84 1\\
85 1\\
86 1\\
87 1\\
88 1\\
89 1\\
90 1\\
91 1\\
92 1\\
93 1\\
94 1\\
95 1\\
96 1\\
97 1\\
98 1\\
99 1\\
100 1\\
101 1\\
102 1\\
103 1\\
104 1\\
105 1\\
106 1\\
107 1\\
108 1\\
109 1\\
110 1\\
111 1\\
112 1\\
113 1\\
114 1\\
115 1\\
116 1\\
117 1\\
118 1\\
119 1\\
120 1\\
121 1\\
122 1\\
123 1\\
124 1\\
125 1\\
126 1\\
127 1\\
128 1\\
129 1\\
130 1\\
131 1\\
132 1\\
133 1\\
134 1\\
135 1\\
136 1\\
137 1\\
138 1\\
139 1\\
140 1\\
141 1\\
142 1\\
143 1\\
144 1\\
145 1\\
146 1\\
147 1\\
148 1\\
149 1\\
150 1\\
151 1\\
152 0.999999999999999\\
153 0.999999999999999\\
154 0.999999999999999\\
155 0.999999999999999\\
156 0.999999999999999\\
157 0.999999999999999\\
158 0.999999999999999\\
159 0.999999999999999\\
160 0.999999999999999\\
161 0.999999999999999\\
162 0.999999999999999\\
163 0.999999999999999\\
164 0.999999999999999\\
165 0.999999999999999\\
166 0.999999999999999\\
167 0.999999999999999\\
168 0.999999999999999\\
169 0.999999999999999\\
170 0.999999999999999\\
171 0.999999999999999\\
172 0.999999999999999\\
173 0.999999999999999\\
174 0.999999999999999\\
175 0.999999999999999\\
176 0.999999999999999\\
177 0.999999999999999\\
178 0.999999999999998\\
179 0.999999999999998\\
180 0.999999999999996\\
181 0.999999999999996\\
182 0.999999999999994\\
183 0.999999999999919\\
184 0.999999999999423\\
185 0.999999999996074\\
186 0.999999999973983\\
187 0.999999999832426\\
188 0.999999998952047\\
189 0.999999993644962\\
190 0.999999962680517\\
191 0.999999788119967\\
192 0.999998839238557\\
193 0.999993878273071\\
194 0.999969011704913\\
195 0.999850030505143\\
196 0.999310095837935\\
197 0.997012128074812\\
198 0.988042423955423\\
199 0.95739991677528\\
200 0.87329565480279\\
201 0.707106781186548\\
202 0.487190619062562\\
203 0.288765301514372\\
204 0.154182257294063\\
205 0.0772451711871729\\
206 0.037139363973817\\
207 0.017318097437692\\
208 0.00787246021866295\\
209 0.00349905935643662\\
210 0.00152365400785586\\
211 0.00065096852348711\\
212 0.000273201327214996\\
213 0.00011273896632449\\
214 4.57810880511899e-05\\
215 1.83071315633366e-05\\
216 7.21336233455454e-06\\
217 2.80199418265696e-06\\
218 1.07352223327795e-06\\
219 4.05834968634566e-07\\
220 1.51441266313666e-07\\
221 5.58009126251093e-08\\
222 2.03082761446577e-08\\
223 7.30230801877191e-09\\
224 2.5948541884616e-09\\
225 9.11449414820242e-10\\
226 3.16528675706672e-10\\
227 1.08702578736242e-10\\
228 3.69232002268287e-11\\
229 1.24066728685844e-11\\
230 4.12463691980625e-12\\
231 1.35733582145616e-12\\
232 4.42003964741545e-13\\
233 1.42742163701434e-13\\
234 4.54773519809507e-14\\
235 1.50687724205528e-14\\
236 5.89983174337997e-15\\
237 5.69881624950702e-15\\
238 5.6003323309543e-15\\
239 5.50264185774455e-15\\
240 5.4252600360708e-15\\
241 5.1830735312815e-15\\
242 5.1368879205954e-15\\
243 4.97629386919343e-15\\
244 4.92401427098856e-15\\
245 4.84999617010169e-15\\
246 4.79805998390163e-15\\
247 4.70579925112866e-15\\
248 4.66657367563156e-15\\
249 4.57186165188267e-15\\
250 4.56419158349445e-15\\
251 4.51714969728416e-15\\
252 4.44897945565394e-15\\
253 4.41297696712668e-15\\
254 4.36775913047172e-15\\
255 4.28796601023895e-15\\
256 4.26883241550268e-15\\
257 4.20198640858813e-15\\
258 4.16066586411662e-15\\
259 4.14502491942751e-15\\
260 4.04752773071054e-15\\
261 4.02002070958001e-15\\
262 3.94604127047529e-15\\
263 3.87563009433251e-15\\
264 3.83578980667469e-15\\
265 3.80666507640145e-15\\
266 3.7969027820344e-15\\
267 3.7424793338949e-15\\
268 3.70087844460706e-15\\
269 3.61869239789062e-15\\
270 3.6006269263034e-15\\
271 3.55097631364902e-15\\
272 3.52239227413196e-15\\
273 3.4869403117794e-15\\
274 3.46278563257072e-15\\
275 3.42369037991773e-15\\
276 3.38107814020415e-15\\
277 3.34956028405386e-15\\
278 3.30421414059321e-15\\
279 3.28383667685452e-15\\
280 3.21840225752796e-15\\
281 3.15067051499966e-15\\
282 3.13252297187441e-15\\
283 3.08737594026692e-15\\
284 3.04780961496623e-15\\
285 3.00818291527548e-15\\
286 2.95638169015464e-15\\
287 2.93489875108495e-15\\
288 2.90476027804772e-15\\
289 2.85911180214696e-15\\
290 2.80235856811339e-15\\
291 2.78538880670818e-15\\
292 2.71932185266321e-15\\
293 2.7028149862441e-15\\
294 2.66414753452827e-15\\
295 2.62919250445293e-15\\
296 2.59251049657372e-15\\
297 2.56699608402979e-15\\
298 2.54487374527012e-15\\
299 2.50489330502225e-15\\
300 2.46312331986186e-15\\
301 2.44142121999655e-15\\
302 2.40276710329481e-15\\
303 2.37509993707391e-15\\
304 2.32760359772636e-15\\
305 2.3057109899816e-15\\
306 2.27921314941029e-15\\
307 2.2432645452491e-15\\
308 2.22943989892751e-15\\
309 2.18497200067638e-15\\
310 2.1803415268201e-15\\
311 2.14670573379456e-15\\
312 2.1116787442456e-15\\
313 2.08489616763686e-15\\
314 2.04828803079519e-15\\
315 2.02093026403333e-15\\
316 2.00257062852027e-15\\
317 1.97577068711085e-15\\
318 1.94066631808673e-15\\
319 1.91035337961431e-15\\
320 1.90261049242461e-15\\
321 1.84960496418279e-15\\
322 1.83352361982122e-15\\
323 1.81591373010123e-15\\
324 1.78676907654867e-15\\
325 1.76868212275333e-15\\
326 1.74769066088856e-15\\
327 1.72188782568121e-15\\
328 1.70335637345011e-15\\
329 1.67396320575549e-15\\
330 1.6646921689456e-15\\
331 1.62479841653213e-15\\
332 1.60095198002813e-15\\
333 1.58228863842419e-15\\
334 1.57204319067562e-15\\
335 1.51027929859025e-15\\
336 1.48555755613872e-15\\
337 1.47923834025061e-15\\
338 1.42468360971492e-15\\
339 1.40876547145308e-15\\
340 1.38621171940967e-15\\
341 1.35182593207901e-15\\
342 1.3338828805722e-15\\
343 1.30401525627695e-15\\
344 1.28700599005213e-15\\
345 1.27004371995624e-15\\
346 1.26650795571463e-15\\
347 1.23584287166474e-15\\
348 1.20799989452959e-15\\
349 1.191364311687e-15\\
350 1.17656238002719e-15\\
351 1.1372878772812e-15\\
352 1.1050297518822e-15\\
353 1.09231046378957e-15\\
354 1.07630217583652e-15\\
355 1.04336872251893e-15\\
356 1.02879715661698e-15\\
357 1.00929804473092e-15\\
358 9.7495379792071e-16\\
359 9.37478937407565e-16\\
360 9.23634426080908e-16\\
361 9.05752765255355e-16\\
362 8.80052494521279e-16\\
363 8.64417961459509e-16\\
364 8.44562774165038e-16\\
365 8.30659505899575e-16\\
366 8.01401604611504e-16\\
367 7.84253680198864e-16\\
368 7.59416611834546e-16\\
369 7.47831645377325e-16\\
370 7.22365379396091e-16\\
371 7.06698251958441e-16\\
372 6.96040288326628e-16\\
373 6.74042645951504e-16\\
374 6.44772161945835e-16\\
375 6.31372938541487e-16\\
376 5.94889800459994e-16\\
377 5.80385176035403e-16\\
378 5.5516643360077e-16\\
379 5.31190354441509e-16\\
380 5.15445612551447e-16\\
381 4.95398957503591e-16\\
382 4.71427138202624e-16\\
383 4.42797777047703e-16\\
384 4.25667544405674e-16\\
385 4.14003375747016e-16\\
386 4.04113827858262e-16\\
387 3.81093296424355e-16\\
388 3.64146415369817e-16\\
389 3.43335810601689e-16\\
390 3.15687630730482e-16\\
391 3.04132458394402e-16\\
392 2.61932606754952e-16\\
393 2.38594712011038e-16\\
394 2.03790083156174e-16\\
395 1.71738790919171e-16\\
396 1.69245914299688e-16\\
397 1.51051096675349e-16\\
398 1.22201371100045e-16\\
399 1.06099499080366e-16\\
400 8.16167269072216e-17\\
401 3.95619830132959e-17\\
};
\addplot [
color=black,
dotted,
forget plot
]
table[row sep=crcr]{
183 1e-18\\
183 10\\
};
\addplot [
color=black,
dotted,
forget plot
]
table[row sep=crcr]{
236 1e-18\\
236 10\\
};
\end{axis}
\end{tikzpicture}%

              };
      \node  (inter) at (3.6,-0.3) {$I_{1-\epsilon}$
            };
      \node  (inter2) at (-0.4,-0.3) {$I_{\chi}$
            };
      \node  (inter3) at (-2.0,-0.3) {$I_\epsilon$
};
    \end{tikzpicture}
    \caption{The subdivision of the spectrum of $A$ into three distinct
      intervals, with cutoff parameter $\tau=1e-14$. The singular values cluster near $1$ and $0$, and there is a plunge region in between. Due to rounding errors, the
      eigenvalues in region $I_\epsilon$ don't decay past machine precision.}
\label{fig:intervals}
\end{figure}
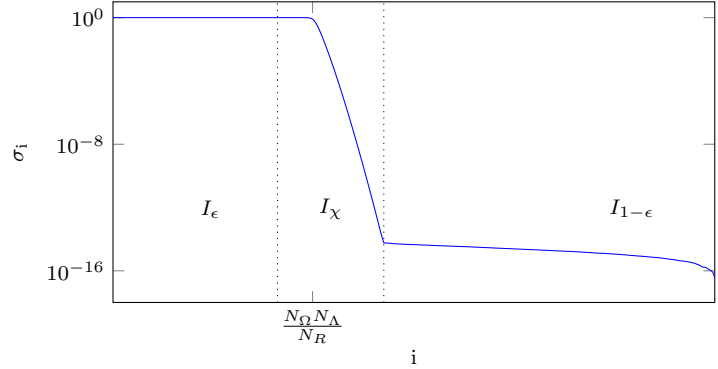

Still, matrix $A$ has a lot of structure that we set out to exploit. The singular values have a very distinct profile, shown in
\cref{fig:intervals}. There is a region $I_{1-\epsilon}$ of singular values that are close to $1$ up to
a small value $\epsilon$, but never exceeding it. Its size, as proven in \cref{sec:spectrum}, is approximately
$\frac{\ndomtp\ndomfp}{\nbboxtp}$. There is a similar region $I_\epsilon$ that contains singular values smaller than $\epsilon$. Inbetween there is a \emph{plunge region} $I_\chi$,
which contains singular values between $\epsilon$ and $1-\epsilon$.

Denote by $\Neps$ the size of the plunge region $I_\chi$,
\begin{equation}
  \label{eq:20}
\Neps=  \#\{\sigma_k:\epsilon<\sigma_k<1-\epsilon\}.
\end{equation}
For the 1D discrete Fourier extension problem this quantity is known to grow
slowly with $\nbboxtp$ \cite{Wilson1987}
\begin{equation}
  \label{eq:26}
  \Neps=O(\log(\nbboxtp)).
\end{equation}
Intuitively, one could state that just $\Neps$ of the singular values capture
almost all the ill-conditioning present in $A$.
This observation led to an algorithm for \cref{eq:lsq} that
is $O(\nbboxtp \Neps^2)=O(\nbboxtp\log^2(\nbboxtp))$ \cite{Matthysen2015}, which we now recall in some detail. Note that due to our definitions, $\ndomfp$ asymptotically grows proportionally to $\nbboxtp$, hence the complexities in terms of $\ndomfp$ and of $\nbboxtp$ are equivalent.

The algorithm of \cite{Matthysen2015} solves the system \cref{eq:lsq} under the assumption that
a solution with sufficiently small norm $\Vert x \Vert$ and small residual $Ax-b$ exists. We refer to \cite{Adcock2012} for conditions on the existence of these solutions in the 1D Fourier extension case, and \cite{Adcock2016} for a detailed treatment in the general context of frames.

\begin{algorithm}[htbp]
\begin{algorithmic}
\STATE Solve $PAy=Pb$ for $y$, with $P = AA' - I$
\STATE $z=A'(b-Ay)$
\STATE $x=y+z$
\end{algorithmic}
\caption{Fourier frame approximations in $O(\ndomfp\Neps^2)$ operations.}
\label{alg:dpss2}
\end{algorithm}

The algorithm is simple to state, and it is included schematically in \cref{alg:dpss2}. Its steps do require some more explanation:
\begin{enumerate}
\item The $\Neps$ singular values in the middle region can be isolated by
  multiplying both $A$ and $b$ with a matrix $P = AA'-I$. This multiplication maps the singular values $\sigma$ of $A$ to $\sigma^3-\sigma$, since if $A = U \Sigma V'$ then $(AA'-I)A = U (\Sigma^3-\Sigma) V'$. This effectively removes all singular values close to either one or zero. As a result, the linear system
  \begin{equation}
    \label{eq:28}
    PAy=Pb
  \end{equation}
has numerical rank $\chi(\epsilon_{mach},\nbboxtp)$, and a solution
$y$ can be obtained in $O(\nbboxtp\log^2(\nbboxtp))$ operations, e.\ g.\ through
randomized algorithms \cite{Liberty2007}.

\item The partial solution vector $y$ has a residual $r = b - Ay$. If we find a vector $z$ such that $Az = r$, then $x = y+z$ solves the overall problem. Indeed, in that case $Ax = Ay + Az = b - r + r = b$.

\item Crucially, the problem for $z$ is simple to solve, $z = A' r = A' (b-Ay)$. If an accurate solution exists, then $r$ must lie in the column space of $A$. Since $r$ is by construction orthogonal to the column space associated with $I_\chi$, it must in fact lie in the space associated with $I_{\epsilon}$. Yet, since the corresponding singular values are close to $1$, in this subspace the inverse of $A$ is well approximated by its adjoint $A'$.
\end{enumerate}

The correctness of the algorithm is shown formally in \cite{Matthysen2015}. For the purpose of the present article, here we make the following observations:
\begin{itemize}
 \item A fast matrix-vector product with $A$ and $A'$ is available in any dimension, regardless of the shape of the domain $\domt$.
 \item Algorithm 1 is purely algebraic and applies to any system $Ax=b$, as long as $A$ has a singular value profile similar to the one shown in \cref{fig:intervals}, exhibiting a plunge region from $1$ to $0$.
 \item The computational cost of the algorithm depends quadratically on the size of the plunge region: it is $O(\ndomfp\Neps^2)$ operations.
\end{itemize}

We show numerical results using Algorithm 1 for a variety of domains in
\cref{sec:experiments}. Compared to its univariate implementation described in \cite{Matthysen2015}, the single conceptual complication lies in the suitable identification of the sampling sets as illustrated in \cref{fig:samplepoints}. Mathematically, however, determining the size of the plunge region -- and with it the computational complexity of our algorithm -- is significantly more involved. The next section contains a brief historical context of the literature on the asymptotic behaviour of $\Neps$, as there has been quite some interest in the equivalent continuous problem.

\section{Spectrum of the collocation matrix}

\label{sec:spectrum}
\subsection{One dimensional bandlimited extrapolation}

We begin this section by highlighting the close interconnection between the
Fourier frame approximation problem and that of bandlimited extrapolation in classical
signal processing literature. Much attention has been given to the problem of
extending or extrapolating a function that is known to be bandlimited, from limited
data. The ingredients of the two problems, a truncated basis and limited data,
are the continuous equivalents of \cref{fig:samplepoints}, and as such it is not surprising that the methods have a lot
in common.

The theory on bandlimited extrapolation was pioneered in a series of
papers by Slepian and collaborators in the 1960s and 1970s \cite{Slepian1978a}
\cite{Landau1961} \cite{Pollack1961} \cite{Pollack1961a}. Specifically, they
studied the integral equation 
\begin{equation}
  \label{eq:integralequation}
\lambda_i\psi_i(s)=  \int_{\Omega}\psi_i(t)\frac{\sin W(t-s)}{\pi(t-s)}dt,
\end{equation}
whose solutions are called \emph{prolate spheroidal wave functions}. This equation corresponds to the question \emph{To what extent $\lambda_i$ can
  a function be concentrated both in the time- and frequency
  domain?}. Indeed, in this equation the function has finite bandwidth $W$ and
the equation expresses that a fraction $\lambda_i$ of the energy of $\psi_i$ is
contained in the subdomain $\Omega\subset\mathbb{R}$. In what is known as the uncertainty principle, or the Gabor limit, they showed that
the eigenvalues approach but never exactly equal $1$, and so
when properly ordered,
\begin{equation}
  \label{eq:2}
  1 > \lambda_1 > \lambda_2 > \dots >0.
\end{equation}
Moreover, the eigenvalues cluster near 1 and zero as $W\to 0$, or equivalently, as
the frequency limit increases. That is, for any small $\epsilon$, the number of
eigenvalues between $1-\epsilon$ and $\epsilon$ grows like $O(\log{W^{-1}})$.

Following up on this result, Slepian defined and
proved similar results for a discrete version of the prolate spheroidal wave
functions \cite{Slepian1978}. In this case the frequency domain is sampled at regular intervals,
and the sinc kernel is replaced by a Dirichlet kernel. \Cref{eq:integralequation}
becomes 
\begin{equation}
  \label{eq:3}
  \lambda_i\psi_i(s)=  \int_{\Omega}\psi_i(t)\frac{\sin W(t-s)}{\sin(\pi(t-s))}dt.
\end{equation}
The eigenfunctions and corresponding Fourier series coefficients are optimally
concentrated in a continuous time and discrete frequency domain, or vice
versa. The clustering property was proven as well.

When both time and frequency domain are discrete, the resulting sequences viewed
on their domain of restriction become finite vectors \cite{Jain1981}. Equation (\ref{eq:2})
becomes a difference equation \cite{Xu1984}:
\begin{equation}
  \label{eq:4}
\sum_{n=0}^M\frac{\sin((2K+1)(m-n)\pi/N)}{N\sin{((m-n)\pi/N)}}
\psi_i[n]=\lambda_i\psi_i[n].
\end{equation}
These \emph{periodic discrete prolate spheroidal sequences $\psi_i[n]$} where later proven
by Wilson to have the same asymptotic scaling of $\eta(\epsilon,N)$ as their continuous
counterparts \cite{Wilson1987}. This result is sufficient to establish the asymptotic complexity of
\cref{alg:dpss2} in one dimension.

Another generalization of the results by Slepian and his collaborators came by
viewing \cref{eq:integralequation} as a special case of a more general \emph{Wiener-Hopf}
operator $(T_\alpha \psi)(\bsx)$
\begin{equation}
  \label{eq:5}
  (T_\alpha \psi)(\bsx)=\left(\frac{\alpha}{2\pi} \right)^D
  \chi_\Lambda(\bsx)\int_\Omega\int_{\Lambda}e^{i\alpha\bsxi\cdot(\bsx-\bsy)}\psi(\bsy)d\bsy
  d\bsxi, \alpha>0.
\end{equation}
This operator in $L^2(\mathbb{R}^d)$ can be reduced to \cref{eq:integralequation} by
taking $\Lambda$ and $\Omega$ intervals in $\mathbb{R}$ and setting $\alpha \sim W^{-1}$. 
Due to the characteristic functions of $\Omega$ and $\Lambda$, these operators have a
discontinuous symbol. Considerable effort
has gone into describing the spectral properties of these operators. 
Starting with Slepian and Pollack \cite{Pollack1961}, the eigenvalue
distribution has been deduced from the trace of functions of the
operators. They showed that for (\ref{eq:integralequation})
\begin{equation}
  \label{eq:7}
  \lim_{\alpha\to\infty} \mbox{Tr}( T_\alpha) = \lim_{\alpha\to\infty}\sum_i^\infty
  \lambda_{\alpha,i} = C_0\alpha+O(1) 
\end{equation}
and 
\begin{equation}
  \label{eq:8}
\lim_{\alpha\to\infty}  \mbox{Tr} (T_\alpha^2) = \lim_{\alpha\to\infty}\sum_i^\infty \lambda_{\alpha,i}^2 = C_0\alpha+C_1\log{\alpha}+O(1).
\end{equation}
Combining these traces with (\ref{eq:2}) a combinatorial argument (which we will return to in Theorem \ref{thm:combinatorial}) shows that $\eta(\epsilon,\alpha)$ only grows as
$\log\alpha$. This technique, utilizing the trace of $T_\alpha-T_\alpha^2$, was later used by Landau and Widom to prove equivalent results when $\Omega$ consists of a finite number of distinct intervals \cite{Landau1980}.

\subsection{Multi-dimensional extensions}

While the one dimensional Prolate Spheroidal Wave functions received considerable interest
in signal processing and mathematics \cite{Osipov2013}, the generalization to multiple
dimensions is not straightforward. Most generalizations are restricted to a setting where $\Omega=\Lambda$, or require at least some structure in both time and frequency domains. In contrast, the most general multidimensional equivalent of (\ref{eq:5}) would be for arbitrary `frequency'
and `time' domains.

Multidimensional equivalents of PSWFs were first considered by Slepian and
Pollack \cite{Pollack1961a}. They proved a double orthogonality property similar to
the one dimensional case, and an eigenvalue distribution as in \cref{eq:2}.
Afterwards, they focused only on the most symmetric case, where both $\Omega$ and
$\Lambda$ are circular. In this case, the symmetry of the problem leads to PSWF
generalizations as a combination of Bessel functions and one-dimensional
PSWFs. Later results were described for rectangular time and frequency domains
\cite{Borcea2008}, or circular frequency regions \cite{Simons2006}. For an overview, see \cite{Simons2011}.

Results on spectral properties for arbitrarily shaped regions appeared in 1982 \cite{Widom1982},
when H. Widom stated a conjecture on the traces of functions of Wiener-Hopf operators
with discontinuous symbols in
higher dimensions. He conjectured that an operator $T_\alpha$ as in \cref{eq:5}
for higher-dimensional $\Omega$ and $\Lambda$ would obey the trace relation
\begin{equation}
  \label{eq:6}
\lim_{\alpha\to\infty}  \mbox{Tr}(T_\alpha-T_\alpha^2) = \alpha^{d-1}\log{\alpha}\mathcal{W}_1(\delta\Lambda,\delta\Omega)+o(\alpha^{d-1}\log{\alpha}).
\end{equation}
Combined with
\begin{equation}
  \label{eq:11}
  \mbox{Tr}(T_\alpha) = \left( \frac{\alpha}{2\pi}\right)^d\int_\Omega\int_\Lambda d\bsxi d\bsx
\end{equation}
this yields a plunge region that grows at least one order slower in
$\alpha$ than the region of ones (up to a log-factor). Moreover, the constant
\begin{equation}
  \label{eq:10}
  \mathcal{W}_1(\Lambda,\Omega) =
  \frac{1}{2(2\pi)^{d+1}}\int_{\delta\Lambda}\int_{\delta\Omega}|\bsn_{\delta\Lambda}(\bsx)\cdot\bsn_{\delta\Omega}|d\bsxi d\bsx
\end{equation}
is dependent only on the geometry of the domains $\Omega$ and
$\Lambda$. This conjecture was proven in 2010 by Sobolev \cite{Sobolev2010} for arbitrary
smooth domains, and the proof was later extended to piecewise continuous
domains \cite{Sobolev2015}.

\subsection{Generalizing discrete Prolate Spheroidal wave sequences}

In light of the work by Slepian et.\ al., and as already alluded to in
\cref{sec:problem} of this paper, the matrix $A$ in our fully discrete setting
can be seen as the composition of three operations in time and frequency:
extending $\domf$ to $\bboxf$ by zeros in the frequency domain, applying a
discrete Fourier transform and restricting the result to $\domt$ in the time
domain. We shall develop this notion more formally.

We introduce several operators, which operate on sequences of length $\nbboxtp$
on an $\singletp\times\singletp\times\dots$ grid. For indexing purposes we
convert \cref{eq:19} to the integer sets
\begin{equation} 
\ibboxtp=\singletp\domtp\qquad\mbox{and}\qquad \idomtp=\singletp\domtp.
\end{equation}
We denote by $T_\domt$ the discrete space-limiting operator that sets all values outside $\domt$ to zero,
\begin{equation}
  \label{eq:37}
  (T_\domt)_{\bsk,\bsl} = \left\{ \begin{array}{cc} 1, & \bsk=\bsl\in\idomtp, \\ 0 & \mbox{otherwise}.\end{array} \right.
\end{equation}
Similarly, the discrete operator $B_\domf$ is an $\nbboxfp \times \nbboxfp$
bandlimiting operator that eliminates all frequency content outside
$\domf$. With $F$ the $D$-dimensional Fourier transform, $B_\domf=FT_\domf F^{*}$. 

With these definitions, the matrix $AA'$ is the nonzero subblock of the operator $\TBT$. Similar to the univariate case in \cite{Matthysen2015}, the entries of $\TBT$ are given in terms of a convolution kernel
\begin{equation}
  \label{eqn:spacelimiting}
  (\TBT)_{\bsk,\bsl}=B(\bsk-\bsl), \qquad\forall \bsk,\bsl\in\idomtp
\end{equation}
where in the multivariate case $B$ is a product of univariate Dirichlet kernels,
  \begin{align}
\label{eqn:bandlimiting}
    B(\bsk) &= \prod_{d=1}^Db(k_d)\\
b(k) &=\frac{\sin(\pi \singlefp k/\singletp)}{\singletp\sin(\pi k/\singletp)}.
  \end{align}
Here, $\bsk=(k_1,k_2,\dots)$ can be a multidimensional point. Recall that $D$ is the number of
dimensions. 

Denote the eigenvectors of the related Hermitian matrix $\BTTB$ by $\phi_i$ and
by $\hat{\phi_i}=T_\domt\phi_i$ the eigenvectors of $\TBT$. The corresponding
eigenvalues of both matrices are the same and denoted by $\lambda_i$. Similar to \cite{Jain1981,Xu1984,Wilson1987}, the following
properties can be shown:
\begin{enumerate}
\item The eigenvalues are bounded above by $1$ and below by $0$.
\item The rank of $\TBT$ and of $\BTTB$ is $\min(\ndomfp,\ndomtp)$.
\item If $\ndomfp<\ndomtp$, the $\phi_i$ are complete in the space of sequences
  bandlimited in $\domf$.
\item
\label{prop:doubleorthogonality}
Define the discrete inner products $\langle  \phi_i,\phi_j\rangle_\bboxt=\phi_i\cdot\phi_j$ and
  $\langle  \phi_i,\phi_j\rangle_\domt=(T_\domt\phi_i)\cdot(T_\domt\phi_j)$. The $\phi_i$ are doubly
  orthogonal with respect to these inner products,
\begin{equation*}
\langle  \phi_i,\phi_j\rangle_\bboxt=\delta_{ij}, \qquad \langle \phi_i, \phi_j\rangle_\domt=\lambda_i\delta_{ij}.
\end{equation*}
\item 
\label{prop:eigenfft}
The $\phi_i(\domt,\domf)$ are eigenvectors of the $D$-dimensional DFT, with
  $\domt$ and $\domf$ interchanged.
  \begin{equation}
    \label{eq:33}
    T_{\domf}F\phi_i(\domt,\domf)=\hat{\phi_i}(\domf,\domt),
  \end{equation}
where $F$ is the $D$-dimensional DFT matrix.
\item Consider the norms corresponding to property \ref{prop:doubleorthogonality} $\Vert \cdot\Vert_\bboxt$ and $\Vert\cdot\Vert_\domt$.
Then among all multidimensional sequences of size $\nbboxtp$ with frequency
support in $\domfp$, $\phi_1$ is the one most concentrated in $\domtp$ with
concentration $\langle \phi_1,\phi_1 \rangle_\domt/\langle  \phi_1,\phi_1
\rangle_\bboxt=\lambda_1$. Similarly, among the sequences of equal frequency support orthogonal to $\phi_1$, $\phi_2$ is the most concentrated in $\domt$. 
\end{enumerate}

\begin{figure}[hbtp]
  \centering
  \subfloat[$\lambda_{60}\sim 1-10^{-10}$]{\includegraphics[scale=0.15]{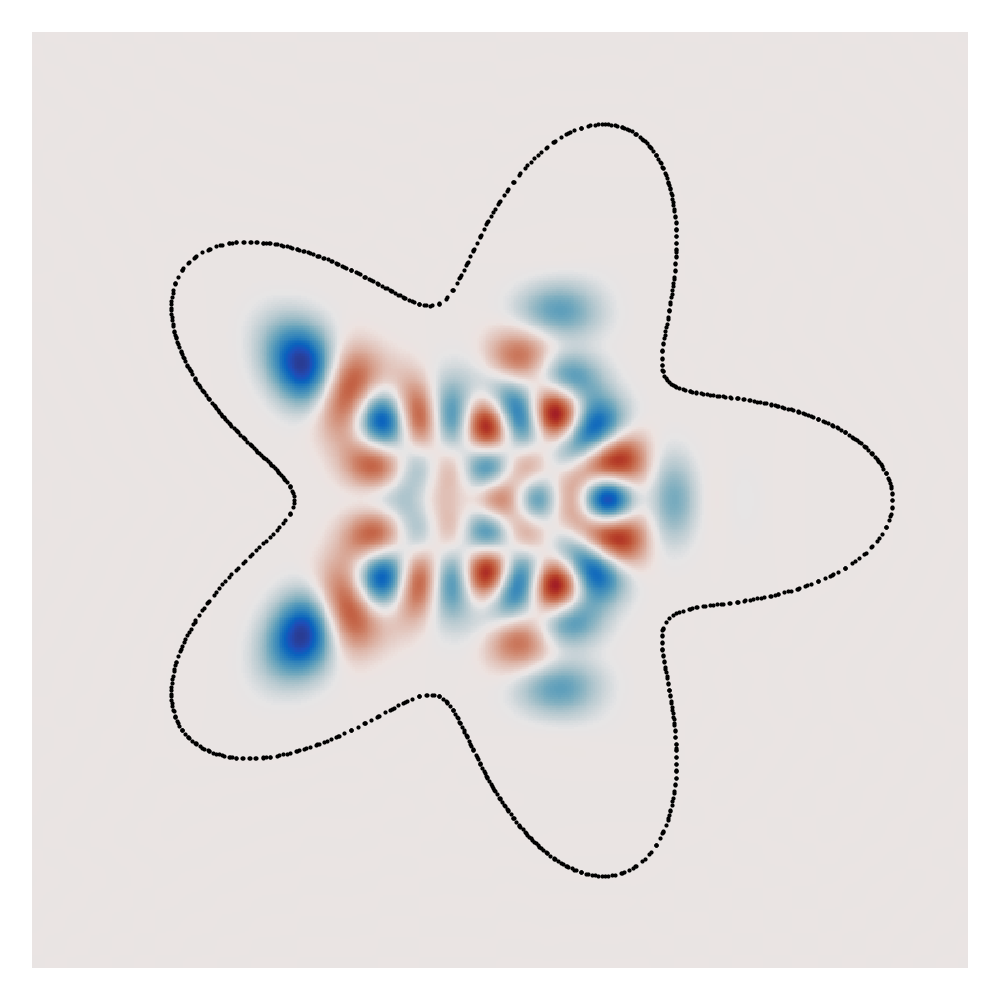}}
  \subfloat[$\lambda_{557}\sim 0.54$]{\includegraphics[scale=0.15]{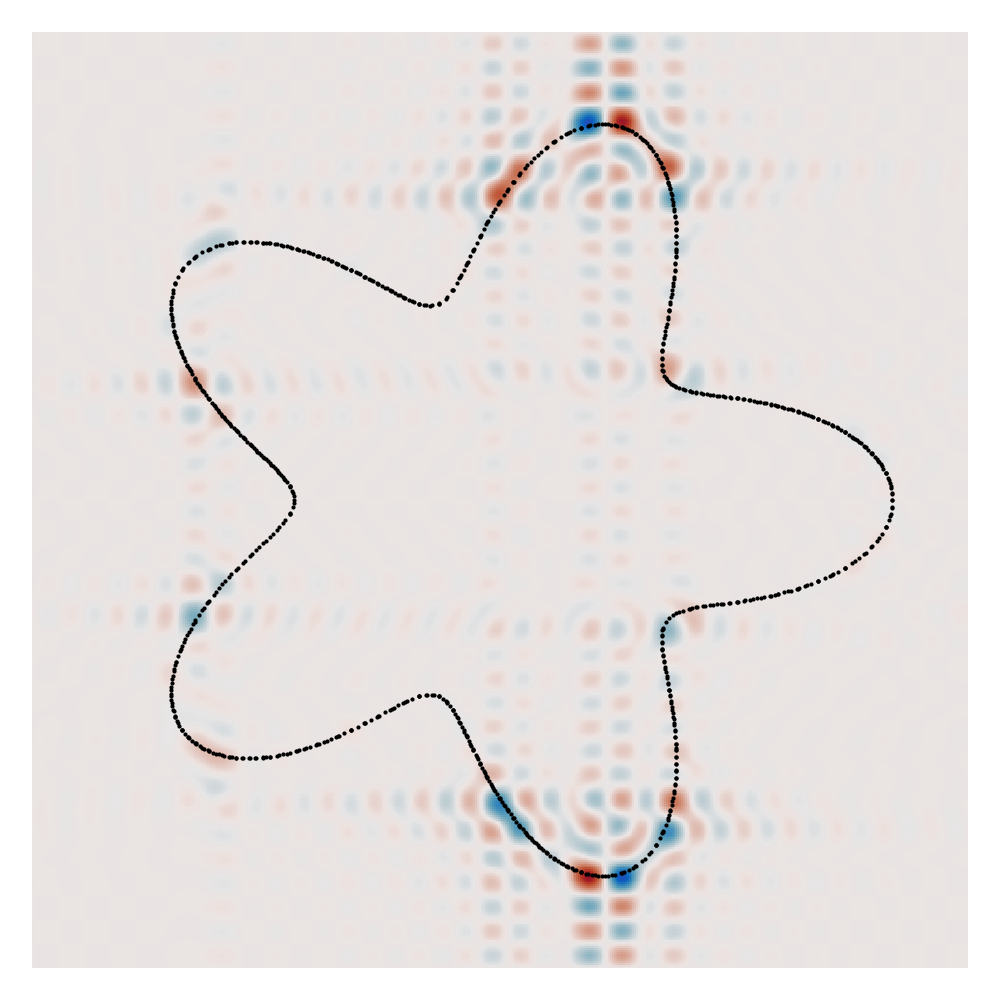}}
  \subfloat[$\lambda_{1300}\sim 10^{-10} $]{\includegraphics[scale=0.15]{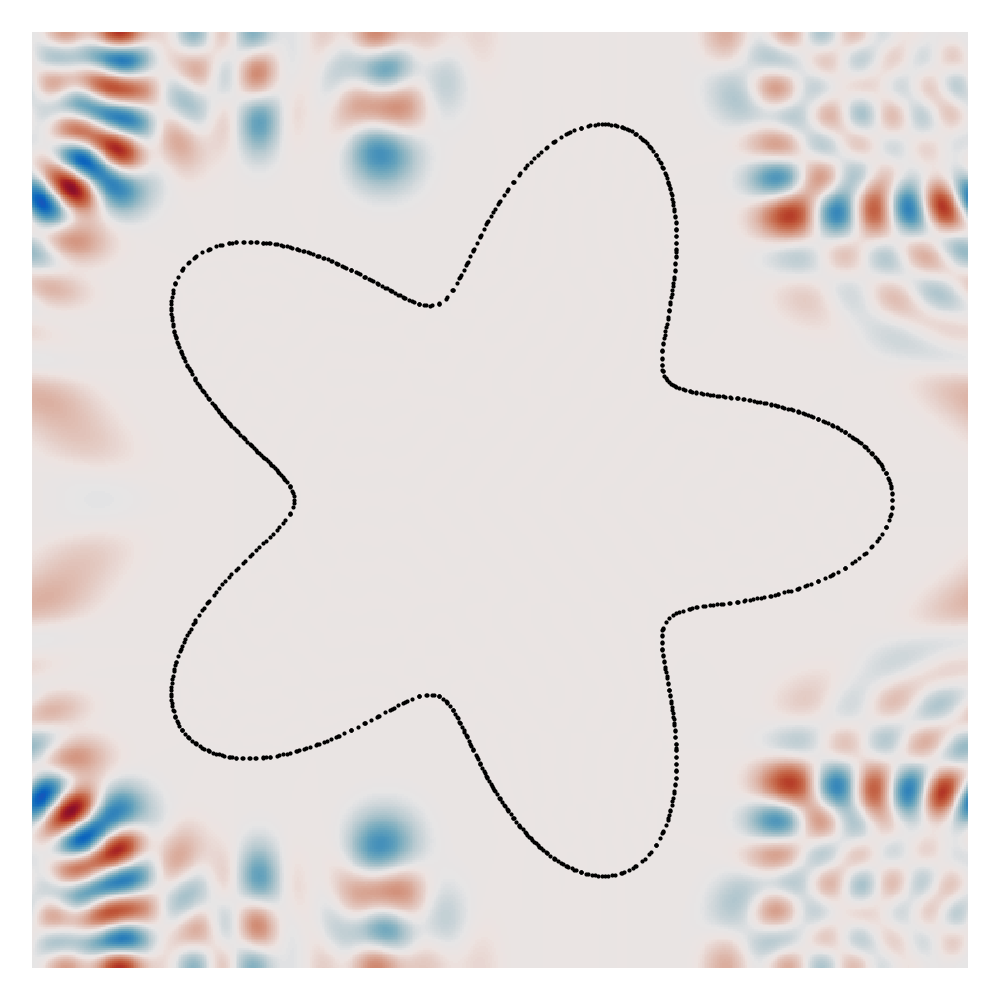}}
  \caption{Fourier series corresponding to periodic discrete prolate spheroidal
    wave sequences $\phi_i$ for different values of the eigenvalue $\lambda_i$. }
  \label{fig:multidpss}
\end{figure}

Let us interpret these properties and make the connection to the algorithm. Recall that the matrix $A$ has a particular singular value structure. The left and right singular vectors of $A$ are the eigenvectors of $AA'$ and of $A'A$ respectively, i.e. they are $\hat{\phi_i}(\domt,\domf)$ and $\phi_i(\domf,\domt)$, with singular values $\sqrt{\lambda_i}$. The singular vectors are, by construction, the periodic discrete prolate spheroidal wave sequences. The vectors can be seen as coefficients of a Fourier series and some of the corresponding functions are shown in \cref{fig:multidpss}.

The maximal ratio $\lambda_1 = \Vert \phi_1 \Vert_{\domf} / \Vert \phi_1 \Vert_{\bboxt}$ with $\lambda_1 \approx 1$ means that $\phi_1$ is almost entirely supported on $\domf$ -- this in spite of being compactly supported in the (discrete) frequency domain. They are, after all, a finite Fourier series. Such a function is shown in the left panel of \cref{fig:multidpss}. In contrast, the functions corresponding to small eigenvalues are almost entirely supported on the exterior domain $\bboxt - \domt$, as shown in the right panel of the figure. Finally, the middle functions with eigenvalues in the plunge region are supported everywhere. This is illustrated in the middle panel. In particular, these functions are the only ones that are non-neglible in a neighbourhood of the boundary. This is a clear indication that the plunge region is a phenomenon that relates to the boundary of the domain at hand.

The solution to $Ax=B$ using a truncated Singular Value Decomposition can be expressed in terms of these generalized discrete Prolate Spheroidal sequences,
\begin{equation}
  \label{eq:38}
  \bsa = \sum_{i=1}^{i_{max}}\frac{1}{\sqrt{\lambda_i}}\hat{\phi_i}(\domf,\domt) \langle f,\hat{\phi_i}(\domt,\domf) \rangle
\end{equation}
where $i_{max}$ is determined by the truncation parameter $\epsilon$ and is such that $\lambda_{i_{max}}\geq\epsilon>\lambda_{i_{max}+1}$. This expression, combined with \cref{fig:multidpss}, clearly illustrates the different steps in Algorithm \ref{alg:dpss2}. This is shown further in \cref{fig:stages}.

\begin{figure}[hbtp]
  \centering
  \subfloat[$b$]{\includegraphics[scale=0.15]{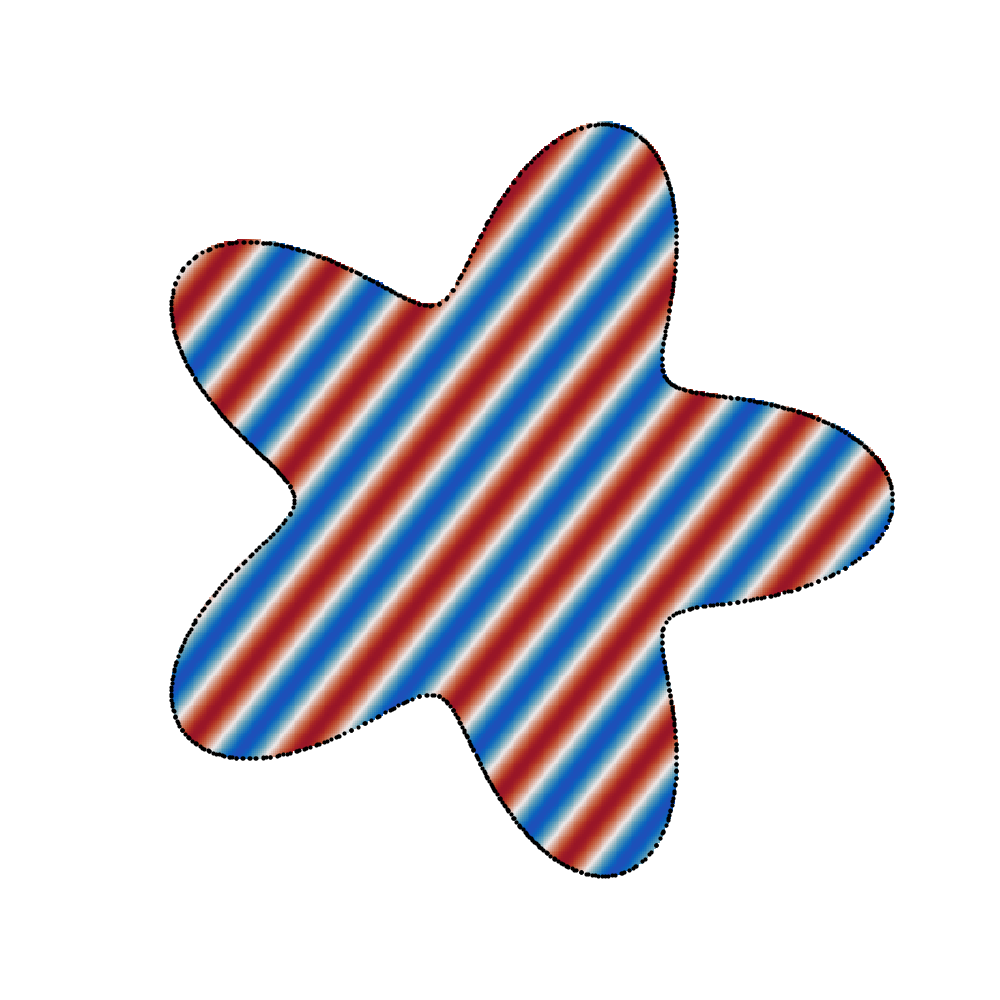}\label{fig:stage0}}
  \subfloat[$Ay, y=V_\chi\Sigma_\chi^{-1}U_\chi'b$]{\includegraphics[scale=0.15]{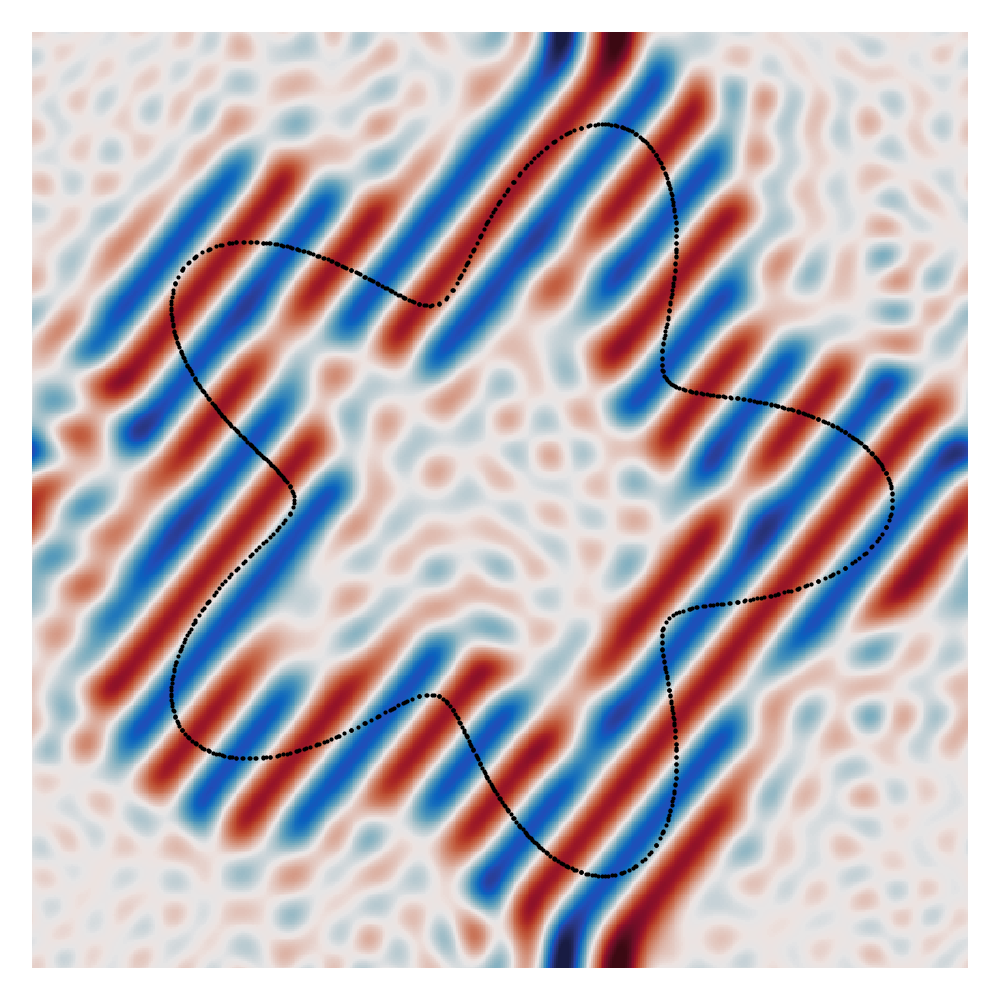}\label{fig:stage1}}
  \subfloat[$Az, z=A'(b-Ay)$]{\includegraphics[scale=0.15]{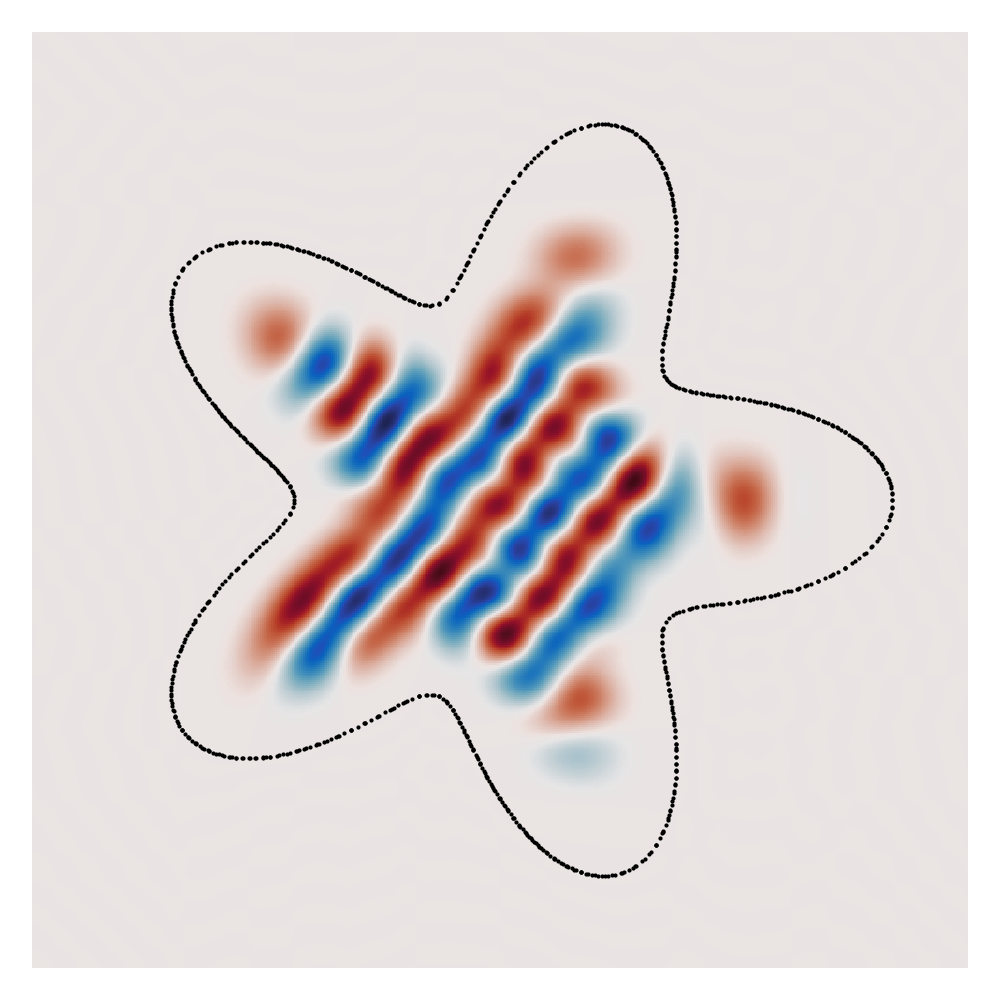}\label{fig:stage2a}}
  \subfloat[$A(y+z)$]{\includegraphics[scale=0.15]{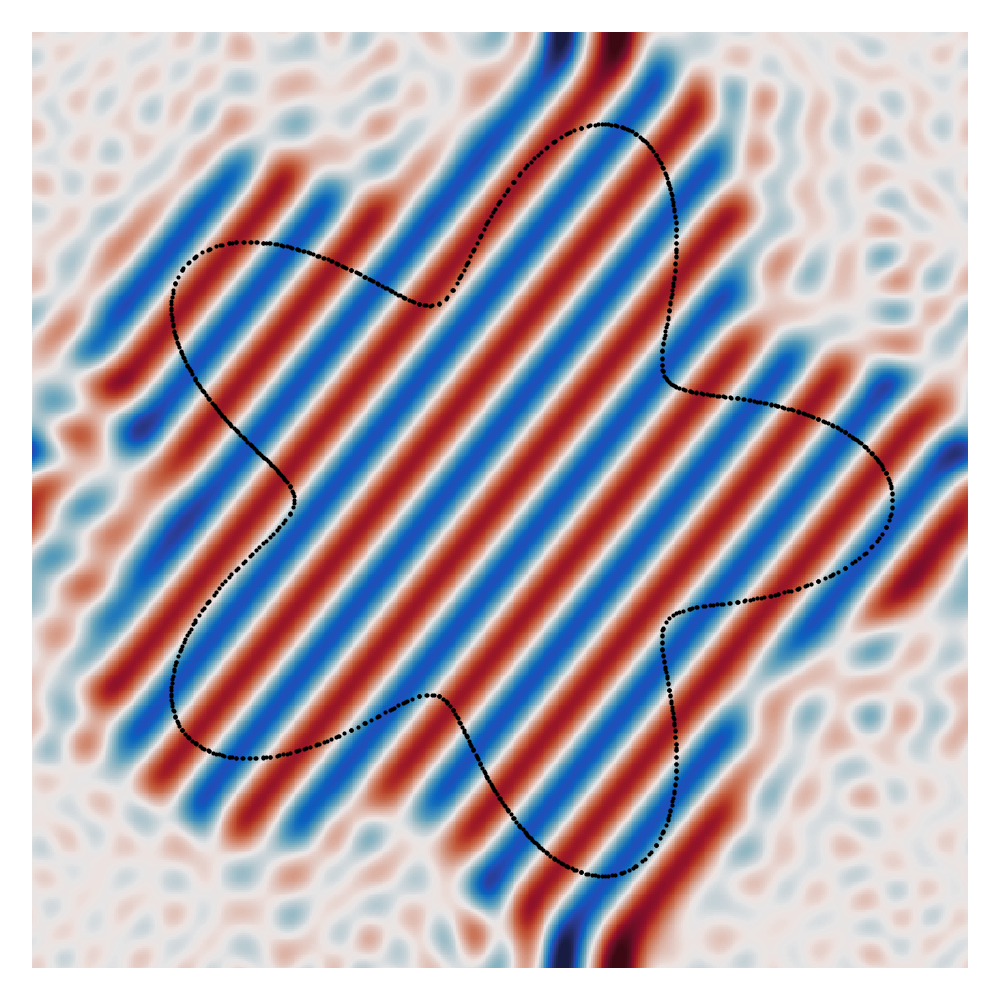}\label{fig:stage2}}
  \caption{Steps in algorithm \ref{alg:dpss2}: Data is given on $\domt$ (\cref{fig:stage0}),
    approximated using the eigenvalues $1-\epsilon>\lambda_i>\epsilon$, and yields
    a good approximation on the boundary (\cref{fig:stage1}). This solution subtracted from the data
(\cref{fig:stage2a}) is easily approximated by a regular Fourier series on the
bounding box (\cref{fig:stage2}). }
  \label{fig:stages}
\end{figure}

The vector $y$ found after the first step is based on the middle singular values, which correspond to functions that are supported along the boundary of the domain. The Fourier series with $y$ as its coefficients is shown in \cref{fig:stage1}: it approximates the data well in a neighbourhood of the boundary. Subtracting this approximation from the original function (as in $b-Ay$) yields a function that vanishes smoothly towards the boundary of $\domt$. Hence, this function can be extended by zero and approximated efficiently with a Fourier transform, and that is expressed by the step $z = A'(b-Ay)$. The vector $z$ is a linear combination of the prolates that are concentrated in the interior of the domain. It is now also clear what the null space of $A$ corresponds to: it consists of linear combinations of the prolates concentrated in the exterior of the domain. Any such prolate can be added to our solution but it will only affect the extension, not the approximation on $\domt$ itself, unless it is multiplied with a very large coefficient.

\subsection{Singular value profile for generalized discrete Prolate Spheroidal
  Sequences}
\label{sec:proof}
Proving asymptotic complexity of algorithm \ref{alg:dpss2} needs a bound on
$\Neps$ as $\ndomfp$ increases. As in \cite{landau,Wilson1987}, this can be inferred from trace iterates of the
operator $\TBT$. After bounding the difference between $tr(\TBT)$ and $tr((\TBT)^2)$, this bound is shown to be of the same order
as $\Neps$. We formulate our final result in Theorem \ref{thm:combined}.

Our bound hinges on two observations:
\begin{itemize}
\item The contribution of a single point in $\domtp$ to $tr(\TBT)-tr((\TBT)^2)$ is inversely proportional to
  the distance between that point and the domain boundary.
\item The number of points at a certain distance from the boundary is bounded by
  the number of boundary points and some terms depending only on domain geometry.
\end{itemize}
The next section contains a concise illustrated proof of the second
observation in the two-dimensional case. The first observation is proven in \Cref{sec:actualproof}. Due to the discrete nature of the problem, we use some concepts known in digital topology \cite{Chen2014}.

\subsubsection{Distance away from the boundary for general 2D domains}
\label{sec:distancebound}

For reasons that will become clear later, the metric of choice is the $l_\infty$
distance,
\begin{equation}
  \label{eq:41}
  d(\bsk,\bsl) = ||\bsk-\bsl||_\infty.
\end{equation}
A point $\bsk$ on a regular grid in two dimensions can have up to $8$ neighbors at a distance $1$. We also assign to each point $\bsk$ a $L_\infty$ distance to the boundary of a point set $P$ (or rather to its exterior),
  \begin{equation}
    \label{eq:1}
    d(\bsk; {})=\underset{\bsl\notin P}{\min} ||\bsl-k||_\infty.
  \end{equation}
Evidently it is true that $\forall \bsk \notin P: d(\bsk; P)=0$, and
\begin{equation}
  \label{eq:30}
  \forall \bsk \in P :   d(\bsk;P) = \left( \underset{\bsl:d(\bsl,\bsk)=1}{\min}d(\bsl;P)\right)+1.
\end{equation}
Next, let $S_i$ denote the points in set $S$ that are a distance $i$ away from the boundary,
\begin{equation}
  \label{eq:45}
  S_i=\{\bsk \in S : d(\bsk; P) = i \}.
\end{equation}
The main result of this section is a bound on the size of these sets, in particular of $|S_{i+1}|$ in terms of $|S_i|$, which
can be obtained using results from digital topology.

Let $\bar{S}_i$ denote the points in set $S_i$ that have no neighbour in $S_{i+1}$
\begin{equation}
  \label{eq:46}
  \bar{S}_i = \{\bsk \in S_{i} :
  \underset{\bsl:d(\bsl,\bsk)=1}{\max}d(\bsl;S)\leq i\}.
\end{equation}
These definitions are illustrated in \Cref{fig:visualillustration}. For example, in the left panel the solid black dots not connected by a line belong to $\bar{S}_1$: their neighbours are either also in $S_1$ or in the exterior of the domain. The black dots connected by a line make up $S_1 \setminus \bar{S}_1$.

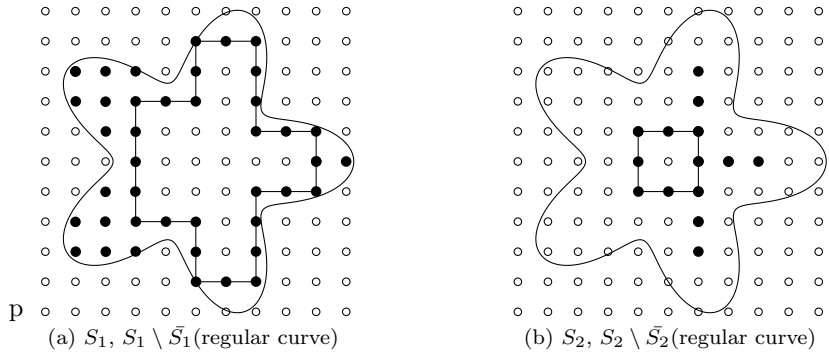
\begin{figure}[hbtp]
  \centering p  
\subfloat[$S_1$, $S_1\setminus\bar{S_1}$(regular curve)]{    \begin{tikzpicture}[
  decoration={
    markings,
    mark=between positions 0 and 1 step 1/42 with {\fill (0pt,0pt) circle (2pt);},
  }
]
    \foreach \x in {-2,-1.6,...,2}{
      \foreach \y in {-2,-1.6,...,2}{
        \draw (\x,\y) circle(0.05cm);
}
}
\draw[domain=0:360,scale=1,samples=500] plot (\x:{1.6+0.5*cos(5*\x)});
\path [postaction={decorate}]
(-1.6,1.2)-- ++(0.8,0) -- ++(0,-0.4) -- ++(0.8,0) -- ++(0,0.8) -- ++(0.8,0) --
++(0,-1.2) -- ++(0.8,0) -- ++(0,-0.4) -- ++(0.4,0) -- ++(-0.4,0) -- ++(0,-0.4)
-- ++(-0.8,0) -- ++(0,-1.2) -- ++(-0.8,0) -- ++(0,0.8) -- ++(-0.8,0) --
++(0,-0.4) -- ++(-0.8,0) -- ++(0,0.4) -- ++(0.4,0) -- ++(0,0.4) -- ++(0.4,0) --
++(0,0.8) -- ++(-0.4,0) -- ++(0,0.4) -- ++(-0.4,0) -- ++(0,0.4);
\fill (-1.6,1.2) circle(0.05cm);
\draw
(-0.8,0.8) -- ++(0.8,0) -- ++(0,0.8) -- ++(0.8,0) --
++(0,-1.2) -- ++(0.8,0) -- ++(0,-0.8) -- ++(-0.8,0) -- ++(0,-1.2) -- ++(-0.8,0) -- ++(0,0.8) -- ++(-0.8,0)
-- ++(0,1.6);
  \end{tikzpicture}}\hspace{2cm}
\subfloat[$S_2$, $S_2\setminus\bar{S_2}$(regular curve)]{\begin{tikzpicture}[
  decoration={
    markings,
    mark=between positions 0 and 1 step 1/20 with {\fill (0pt,0pt) circle (2pt);},
  }
]
    \foreach \x in {-2,-1.6,...,2}{
      \foreach \y in {-2,-1.6,...,2}{
        \draw (\x,\y) circle(0.05cm);
}
}
\draw[domain=0:360,scale=1,samples=500] plot (\x:{1.6+0.5*cos(5*\x)});
\path [postaction={decorate}]
(-0.4,0.4) -- ++(0.8,0) -- ++(0,0.8) -- ++(0,-1.2) -- ++(0.8,0) -- ++(-0.8,0) --
++(0,-1.2) -- ++(0,0.8) -- ++(-0.8,0) -- ++(0,0.8);
\draw
(-0.4,0.4) -- ++(0.8,0) -- ++(0,-0.8) -- ++(-0.8,0) -- ++(0,0.8);
  \end{tikzpicture}}\hspace{2cm}

  \caption{An illustration of the sets $S_1$, $\bar{S_1}$ and their difference $S_1 \setminus \bar{S}_1$ in a component without holes (left panel), and similarly for $S_2$. It is clear that $|S_1|\geq |S_1\setminus\bar{S_1}| > |S_2| \geq |S_2\setminus\bar{S_2}|$. The set $S_3$ in this example consists of a single point.}
  \label{fig:visualillustration}
\end{figure}

Following the terminology of \cite{Chen2014}, we define a line cell as an adjacent pair of points $(\bsk,\bsl): d(\bsk,\bsl)=1$ and a surface cell as a set of four points where all pairwise distances are $1$. We say that a pair of surface cells is point connected if they share a point, and line-connected if they share two points.  

This allows us to state the definition of a regular digital manifold.

\begin{mydef}\cite[Definition 5.14]{Chen2014}
A point set $S$ on a rectangular grid is a regular digital manifold if 
\begin{itemize}
\item all points belong to a surface cell;
\item for any pair of surface cells there is a line-connected path between them.
\end{itemize}
\end{mydef}

We also define a slightly broader class of digital manifolds:
\begin{mydef}
A set $S$ is a pseudoregular digital manifold if 
\begin{itemize}
\item all points belong to a surface cell;
\item for any pair of surface cells there is a point-connected path between them.
\end{itemize}
\end{mydef}

We have the following theorem.
\begin{thm}
 For any finite 2-dimensional set $S$,  $S-\bar{S}_1$ is a finite union of pseudoregular digital manifolds.
\end{thm}
\begin{proof}
If $S-\bar{S}_1$ is empty, then the result is true. Henceforth we assume it is not empty. To prove the first requirement of a pseudoregular manifold, note that because of \cref{eq:30} and \cref{eq:46}, every $\bsk\in S_2$ is surrounded by points in $S-\bar{S}_1$ and is therefore part of $4$ surface cells.  Furthermore, because of \cref{eq:46} every point in $S_1-\bar{S_1}$ has at least
one neighbor in $S_2$, and is therefore part of a surface cell. Grouping surface cells by point-connectedness, the result is a union of pseudoregular manifolds.
\end{proof}

\begin{thm}
  The distance to the boundary is preserved for all points after the removal of
  $\bar{S}_1$,
  \begin{equation}
    \label{eq:47}
    \forall \bsk \in S\setminus \bar{S}_1 : d(\bsk;S-\bar{S}_1)=d(\bsk;S).
  \end{equation}
\end{thm}
\begin{proof}
First note that the distance of a point is the minimum over all the $8$ connected
neighbours plus one. Therefore if $S_i$ stays the same, $S_{i+1}$ stays the same.
Then note that all neighbors of points in $S_2$ are retained in $S-\bar{S}_1$.
\end{proof}

\begin{thm}\cite[Theorem 5.4]{Chen2014}
  The boundary $\delta S$ of a regular 2-dimensional manifold $S$ is itself regular, and the
  union of closed regular curves.
\end{thm}
\begin{thm}\cite[Lemma 9.1]{Chen2014}
  A closed 2-dimensional digital curve has 4 more convex corners than non-convex corners.
\end{thm}

The combination of these two theorems leads to
\begin{thm}
\label{thm:convexconcave}

  For a 2-dimensional pseudoregular manifold
  \begin{equation}
    \label{eq:29}
    |S_{i+1}|\leq |S_i|-4, \qquad i \geq 1.
  \end{equation}
\end{thm}
\begin{proof}
Consider a regular manifold $S$. All points in $S_1$ form a closed digital curve
with 4 more convex corners than non-convex corners. As illustrated in
\cref{fig:proofillustration}, a point on a straight segment maps to one element
of $S_2$, a convex corner maps three points of $S_1$ to one of $S_2$, and a non-convex
corner maps one points of $S_1$ to three of $S_2$. Since the points that are being mapped to 
can also coincide, the bound \cref{eq:29} is obtained.

For a pseudoregular manifold, note that a pair of point connected components can
be regarded as 2 regular manifolds, with one point in common. Combining the
bounds for both regular manifolds and subtracting the one point in common we end
up with \cref{eq:29}. Then the theorem can be applied recursively by removing $\bar{S}_{i+1}$,
obtaining another pseudoregular set.
\end{proof}

For an illustration of this proof, see \cref{fig:proofillustration}.
\begin{figure}[hbtp]
  \centering
\subfloat[straight]{
    \begin{tikzpicture}
\node (whatever) at (-0.5,0) {};
\node (whatever) at (1.5,0) {};
\node[circle,color=black,inner sep=0pt,fill=black,minimum size=4pt] (a1) at (0,-1) {};
\node[circle,color=black,fill=black,inner sep=0pt,minimum size=4pt] (b1) at (0,0) {};
\node[circle,color=black,fill=black,inner sep=0pt,minimum size=4pt] (c1) at (0,1) {};
\node[circle,draw=black,fill=white, inner sep=0pt,minimum size=4pt] (a2) at (1,-1) {};
\node[circle,draw=black,fill=white, inner sep=0pt,minimum size=4pt] (b2) at (1,0) {};
\node[circle,draw=black,fill=white, inner sep=0pt,minimum size=4pt] (c2) at (1,1) {};
\draw (a1)--(b1)--(c1);
\draw (a2)--(b2)--(c2);
\draw[style=dashed, -latex] (b1)--(b2);
  \end{tikzpicture}}\hspace{2cm}
\subfloat[convex]{
    \begin{tikzpicture}
\node[circle,color=black,inner sep=0pt,fill=black,minimum size=4pt] (a1) at (0,-1) {};
\node[circle,color=black,fill=black,inner sep=0pt,minimum size=4pt] (b1) at (0,0) {};
\node[circle,color=black,fill=black,inner sep=0pt,minimum size=4pt] (c1) at
(0,1) {};
\node[circle,color=black,fill=black,inner sep=0pt,minimum size=4pt] (d1) at
(1,1) {};
\node[circle,color=black,fill=black,inner sep=0pt,minimum size=4pt] (e1) at (2,1) {};
\node[circle,draw=black,fill=white, inner sep=0pt,minimum size=4pt] (a2) at (1,-1) {};
\node[circle,draw=black,fill=white, inner sep=0pt,minimum size=4pt] (b2) at (1,0) {};
\node[circle,draw=black,fill=white, inner sep=0pt,minimum size=4pt] (c2) at (2,0) {};
\draw (a1)--(b1)--(c1)--(d1)--(e1);
\draw (a2)--(b2)--(c2);
\draw[style=dashed, -latex] (b1)--(b2);
\draw[style=dashed, -latex] (c1)--(b2);
\draw[style=dashed, -latex] (d1)--(b2);
  \end{tikzpicture}}\hspace{2cm}
\subfloat[non-convex]{
    \begin{tikzpicture}
\node[circle,color=black,inner sep=0pt,fill=black,minimum size=4pt] (a1) at (-1,0) {};
\node[circle,color=black,fill=black,inner sep=0pt,minimum size=4pt] (b1) at (0,0) {};
\node[circle,color=black,fill=black,inner sep=0pt,minimum size=4pt] (c1) at
(0,1) {};
\node[circle,draw=black,fill=white, inner sep=0pt,minimum size=4pt] (y2) at
(-1,-1) {};
\node[circle,draw=black,fill=white, inner sep=0pt,minimum size=4pt] (z2) at (0,-1) {};
\node[circle,draw=black,fill=white, inner sep=0pt,minimum size=4pt] (a2) at (1,-1) {};
\node[circle,draw=black,fill=white, inner sep=0pt,minimum size=4pt] (b2) at (1,0) {};
\node[circle,draw=black,fill=white, inner sep=0pt,minimum size=4pt] (c2) at (1,1) {};
\draw (a1)--(b1)--(c1);
\draw (y2)--(z2)--(a2)--(b2)--(c2);
\draw[style=dashed, -latex] (b1)--(b2);
\draw[style=dashed, -latex] (b1)--(a2);
\draw[style=dashed, -latex] (b1)--(z2);
  \end{tikzpicture}}
  \caption{Illustration accompanying \cref{thm:convexconcave}.}
  \label{fig:proofillustration}
\end{figure}
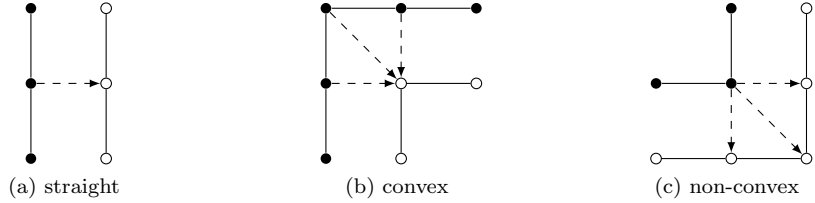

We conclude with a generalization that allows for a finite number of holes in a set. The set $S_{i+1}$ may be larger in this case than $S_i$, but the small growth does not invalidate the asymptotic complexity in the next section.

\begin{thm}
  For a 2-dimensional set containing $c$ 8-connected components and $h$ holes, the number of
  points a distance $i$ away from the boundary is bounded by
  \begin{equation}
    \label{eq:18}
    |S_{i+1}| \leq |S_i| - 4(c-h), \qquad i \geq 1.
  \end{equation}
\label{thm:withholes}
\end{thm}
\begin{proof}
For $c$ 8-connected components, \cref{thm:convexconcave} holds individually for each $S_{ij}$. Thus the bound for the combined sets $S_i$ is 
\begin{equation}
    |S_{i+1}| \leq |S_i| - 4c.
\end{equation}
A hole in this context is a connected component not in $P$ but entirely surrounded by it. Denote by $S_{iB}$ the points whose closest neighbor not in $P$ is in the hole. Then a similar reasoning to \cref{thm:convexconcave} shows that 
\begin{equation}
|S_{i,B}|\leq |S_i\setminus\bar{S}_{i,B}| < |S_{i+1,B}|+4.
\end{equation}
Summing the bounds completes the proof.
\end{proof}

\begin{remark}
 Theorem \ref{thm:convexconcave} does not hold in three dimensions and higher. In fact, the set $S_{i+1}$ can be larger than $S_i$ even for domains without a hole. A domain with an intrusion can have interior non-convex angles, at which a single point in $S_i$ maps to many points in $S_{i+1}$.
\end{remark}

\subsubsection{Bounding  $\Neps$}
\label{sec:actualproof}

\begin{thm}
\label{thm:traces}
Let $T_\domt$ and $B_\domf$ be as in
\cref{eqn:spacelimiting,eqn:bandlimiting}. We are interested in the behavior for
large $\singlefp$, with constant oversampling $\gamma=\singlefp/\bboxfp$. Furthermore, let $\ndombp(\singlefp)$ denote
the number of points in $P_{\domb}$ neighbouring the boundary, i.\e.\ $S_1$ from the previous section:
\begin{equation}
  \label{eq:39}
  P_{\domb} = \{\bsk\in\domt\quad | \quad\exists \bsl,||\bsl||_\infty=1: \bsk+\bsl \notin \domt\}.
\end{equation}
We further assume that the limit
\begin{equation*}
\underset{\singlefp\to\infty}{\lim}(h(\domtp)-c(\domtp))=C
\end{equation*}
exists with $C < \infty$, where $h(\domtp)$ and $c(\domtp)$ are as before the number of holes and distinct connected components of $\domtp$.
  Then for the operator $\TBT$ 
\begin{equation}
  \label{eq:23}
  \underset{\singlefp\to\infty}{\lim}tr(\TBT)-tr((\TBT)^2) = O\left(\ndombp\log{\singlefp}\right).
\end{equation}
\end{thm}
\begin{proof}
The trace of $\TBT$ is
\begin{align}
  tr(\TBT) &= \sum_{\bsk\in\idomtp} B(\bsk-\bsk) \\
 &= \ndomtp B(\bszero)=  \frac{\ndomtp\ndomfp}{\nbboxtp}.
\label{eqn:trace1}
\end{align}
For the squared operator trace, note that 
\begin{equation}
tr((\TBT)^2)=||\TBT||_F=\sum_{\bsk\in\idomtp}\sum_{\bsl\in\idomtp}|(\TBT)_{\bsk,\bsl}|^2.
\label{eqn:tracesquared}
\end{equation}
Now, define an intermediate function
\begin{align*}
    f(\bsk) &= \sum_{\bsl\in\idomtp} |(\TBT)_{\bsk,\bsl}|^2,\qquad tr((\TBT)^2)=\sum_{\bsk\in\idomtp}f(\bsk).
\end{align*}
This $f$ can be rewritten as
\begin{align*}
  f(\bsk)&=\sum_{\bsk\in \idomtp} |B(\bsk-\bsl)|^2\\
&= \sum_{\bsl\in \ibboxtp} |B(\bsk-\bsl)|^2-\sum_{\bsl\in (\ibboxtp\setminus\idomtp)}|B(\bsk-\bsl)|^2
\end{align*}
The first sum is equal to $\frac{\ndomfp}{\nbboxtp}$ through Parseval's equation. The
second term is the sum over the index set $\ibboxtp \setminus\idomtp$. As
a shorthand notation, use 
\begin{equation}
q_{\bsk}=d(\bsk;\idomtp).
\end{equation}
The largest inscribed square around $\bsk$ is then given by $\bsk+Q_\bsk\times
Q_\bsk$, see \cref{fig:proofillustration2}.
with $Q_\bsk=\{-q_\bsk+1,\dots,q_\bsk-1\}$. Restricting $\idomtp$ to this square
and using that due to periodicity
$\sum_{\bsl\in\ibboxtp}|B(\bsl)|^2=\sum_{\bsl\in\ibboxtp-\bsk}|B(\bsl)|^2$, the
last sum can be bounded by

\begin{align}
  \sum_{\bsl\in(\ibboxtp\setminus\idomtp)}|B(\bsk-\bsl)|^2&<\sum_{\bsl\in(\ibboxtp\setminus
                                                         (\bsk+Q_\bsk\times Q_\bsk)}|B(\bsk-\bsl)|^2\\
&=\sum_{\bsl\in(\ibboxtp\setminus (Q_\bsk\times Q_\bsk)}|B(\bsl)|^2\\
  \label{eqn:prooftpstep}&=\left(\sum_{k\in R_d\setminus Q_{\bsk}}|B(k)|^2_d\right)^2+2\left(\sum_{k\in R_d\setminus Q_{\bsk}}|B(k)|^2\sum_{k\in
  Q_{\bsk}}|B(k)|^2\right).
\end{align}
Here $R_d=\{0,\dots,\singletp-1\}$, and $B(k)$ is the one-dimensional kernel.
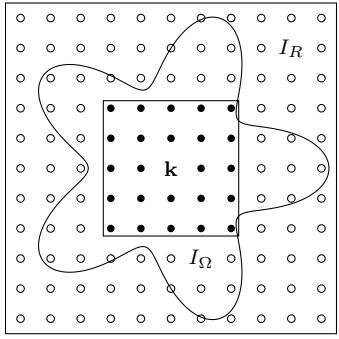
\begin{figure}[hbtp]
  \centering   
    \begin{tikzpicture}
    \foreach \x in {-2,-1.6,...,2}{
      \foreach \y in {-2,-1.6,...,2}{
        \pgfmathparse{max(abs(\x),abs(\y))<1 ? int(1) : int(0)}
        \ifnum\pgfmathresult=1
           \fill (\x,\y) circle(0.05cm);
        \else
           \draw (\x,\y) circle(0.05cm);
        \fi
}
}
\draw[domain=0:360,scale=1,samples=500] plot (\x:{1.6+0.5*cos(5*\x)});
\fill[color=white] (0.4,-1.2) circle(0.3cm);
\node (pomega) at (0.4,-1.2) {$\idomtp$};
\fill[color=white] (1.6,1.6) circle(0.3cm);
\node (pr) at (1.6,1.6) {$\ibboxtp$};
\fill[color=white] (-0.0,0.0) circle(0.3cm);
\node (pr) at (-0.0,0.0) {$\bsk$};
\draw (-0.9,-0.9) rectangle (0.9,0.9);
\draw (-2.2,-2.2) rectangle (2.2,2.2) ;
  \end{tikzpicture}
  \caption{The largest inscribed square in $\idomtp$ around any point $\bsk$ is
    $\bsk+Q_\bsk\times Q_\bsk$. In this figure $q_\bsk=3$, leading to a
    $5\times 5$ square.}
  \label{fig:proofillustration2}
\end{figure}

From \cite{Wilson1987}, the first sum can be bounded by
\begin{equation}
  \label{eq:24}
  \sum_{k\in
  R_d\setminus Q_{\bsk}}|B(k)|^2_d=\sum_{k=q_\bsk}^{\singletp-q_\bsk}\left(\frac{\sin(\pi k / \os)}{\singletp \sin(\pi k/\singletp)}\right)^2<\frac{1}{4q_\bsk}+\frac{\os}{16 q_\bsk^2},
\end{equation}
Further, $\sum_{k\in
  Q_{\bsk}}|B(k)|^2_d< \os$. Then \cref{eqn:prooftpstep} can be bounded by a rational polynomial in $q_\bsk$.
\begin{align}
  \label{eq:12}
  \sum_{\bsl\in
  (\ibboxtp\setminus\idomtp)}|B(\bsk-\bsl)|^2&<\frac{\os}{2} q_\bsk^{-1}+\left(\frac{\os^2}{2^3}+\frac{1}{2^4}\right)q_\bsk^{-2}+\frac{\os}{2^5} q_\bsk^{-3}+\frac{\os^2}{2^8}q_\bsk^{-4},
\end{align}
with all coefficients independent of $\nbboxtp$.
Then 
\begin{align*}
  tr((TBT)^2) &= \sum_{\bsk\in\idomtp} f(\bsk) \\
&> \frac{\ndomtp \ndomfp}{\nbboxtp}-\sum_{\bsk\in\idomtp}\left(\os q_\bsk^{-1}+O(q_{\bsk}^{-2})\right).
\end{align*}
Now recall from \cref{sec:distancebound} that $\idomtp$ can be divided into sets
$S_i=\{\bsk:q_\bsk=i\}$. \Cref{thm:withholes} states that $|S_{i+1}|<|S_i|-4(c-h)$.
Furthermore, the size of the bounding box dictates that $q_{\bsk}$ can never exceed $\frac{\singletp}{2}$. With this in mind it is easier to sum over the regions $S_i$ than over all points at once.
This leads to a bound
\begin{align}
\frac{\ndomtp\ndomfp}{\nbboxtp}-tr((TBT)^2)&<\sum_{i=1}^{\singletp/2}\sum_{S_i}\left(\os i^{-1}+O(i^{-2})\right)\\
&<\sum_{i=1}^{\singletp/2}\left(\ndombp+4i(h-c)\right)\left(\os i^{-1}+O(i^{-2})\right)\\
\label{eqn:squaredtrace}
&<C_1\ndombp\log{\singletp}+C_2\ndombp
\end{align}
\Cref{eqn:trace1,eqn:squaredtrace} combined give the desired result.
\end{proof}

Next, we want to relate the difference between iterated traces to the plunge region. This relies on a fairly general counting argument. Recall that the trace of a matrix equals the sum of its eigenvalues, and the trace of a matrix squared equals the sum of the squares of its eigenvalues.

\begin{thm}
\label{thm:combinatorial}
Let $1>\lambda_1(N)>\lambda_2(N)>\dots>\lambda_N(N)>0$ be a given ordered series where 
\begin{align}
  \label{eq:21}
  \sum_{i=1}^N\lambda_i(N) &= CN \\
  \label{eq:21b}
\sum_{i=1}^N\lambda_i(N)^2 &= CN-g(N)
\end{align}
where $g(N)=o(N)$ is a positive function. Then $|\{\lambda_k:\epsilon<\lambda_k<1-\epsilon\}|=O(g(N))$.
\end{thm}
\begin{proof}
Define $k_{\min}$ and $k_{\max}$ as the limits of the intermediate region
\begin{equation}
  \label{eq:13}
  k_{\min}=\argmin_k \lambda_k:\lambda_k < 1-\epsilon, \qquad k_{\max} = \argmax_k \lambda_k:\lambda_k > \epsilon.
\end{equation}
Then $\forall k>k_{\min}:\lambda_k^2<(1-\epsilon)\lambda_k$ and $\forall k\leq
k_{\min}: \lambda_k^2<(1-\epsilon)\lambda_k+\epsilon $, so that
\begin{equation}
  \label{eq:14}
  \sum_k \lambda_k^2 <(1-\epsilon)\sum_k\lambda_k + \epsilon k_{\min}.
\end{equation}
Substituting \cref{eq:21,eq:21b} leads to
\begin{equation}
  \label{eq:25}
  k_{\min} > CN-\frac{g(N)}{\epsilon}.
\end{equation}
Similarly, $\forall k<k_{\max}:\lambda_k^2<(1+\epsilon)\lambda_k-\epsilon$ and
$\forall k\geq k_{\max}:\lambda_k^2<(1+\epsilon)\lambda_k$, so that
\begin{equation}
  \sum_k\lambda_k^2<(1+\epsilon)\sum_k\lambda_k-\epsilon k_{\max}.
\end{equation}
Combined with \cref{eq:21,eq:21b} this yields the upper bound
\begin{equation}
  \label{eq:27}
  k_{\max}<CN+\frac{g(N)}{\epsilon}
\end{equation}
\end{proof}


\begin{thm}
\label{thm:combined}
Let $\domt,T_\domt,B_\domf$ and $P_{\domb}$ be as in Theorem \ref{thm:traces}. Then for the operator $T_{\domtp} B_{\domf} T_{\domtp}$ 
\begin{equation}
  \label{eq:theorem3}
  \Neps = O\left(\ndombp\log{\singletp}\right)
\end{equation}
where $\Neps$ is as in \cref{eq:20}.
\end{thm}
\begin{proof}
The proof follows directly from \cref{thm:traces} and \cref{thm:combinatorial},
and noting that for square matrices $Tr(A^k)=\sum_k \lambda^k$.
\end{proof}

\begin{remark}
\Cref{thm:combined} gives a bound in terms of $\ndombp(\singlefp)$. For any
2-dimensional non-fractal
domain, $\ndombp=\mathcal{O}(\singlefp)$. To see this, note that
$d=\lim_{\singlefp\to\infty}\frac{\log{\ndombp}}{\log{\singlefp}}$ is equal to
the box-counting or Minkowski-Bouligand definition of the boundary dimension \cite{Falconer1990}. For a non-fractal
domain this is equal to the topological dimension of the boundary which is 1.
\end{remark}


\begin{remark}
As mentioned in the introduction, \Cref{thm:combined} leads to an
$O(\ndomfp^2\log(\ndomfp)^2)$ complexity for \cref{alg:dpss2} on 2D domains. It is however difficult to extend the results from
\cref{sec:distancebound} to higher dimensions, as there is no straightforward equivalent of
\cref{thm:convexconcave}. If extended to those domains
\cref{thm:combined} yields asymptotic reductions in higher
dimensions, though the savings have diminishing returns, generally of the order
$O(N^{3d-2})$ versus $O(N^{3d})$ for a full SVD.
\end{remark}

\section{Numerical results}
\label{sec:experiments}
This section contains examples and numerical results for various two-dimensional domains and
method parameters. The aim is to demonstrate the asymptotic complexity, convergence
properties and robustness of the algorithm.
Algorithm \ref{alg:dpss2} was implemented using a randomized algorithm for the
low rank subsystem.
A selection of possible geometries was used, shown in 
\cref{fig:domains}. All domains are normalized to have equal area. These domains
where chosen to showcase differences in results for contrasting properties:
\begin{itemize}
\item The square and diamond show the method is not rotation invariant.
\item The square and disk show the effect of corners
\item The disk and ring show the effect of a simply connected domain versus a not
  simply connected domain. 
\item A double asteroid is included to study boundaries that are not smooth.
\end{itemize}
The precise effect of the domain on complexity and accuracy is discussed in
\cref{sec:domaininfluence}.

Throughout these experiments, the basis of choice is a Fourier basis on the
rectangle $[-T,T]\times [-T,T]$ with $N^2$ degrees of freedom. Unless
specified otherwise, the value for $T$ is 2 and the oversampling factor
$\ndomtp/\ndomfp$ is taken to
be 4. The cutoff $\epsilon$ (expressed through the estimate of $\Neps$) is
consistently $10^{-14}$.
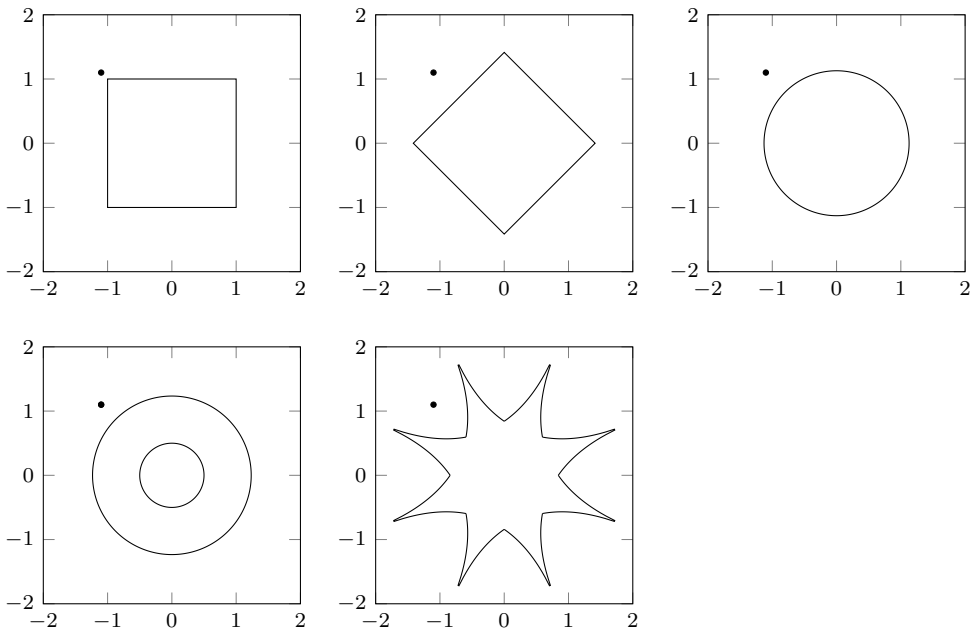
\begin{figure}[hbt]
  \centering
\begin{tikzpicture}
\begin{groupplot}[
    group style={group size=3 by 2},
    axis lines=box,
    axis equal image,
    height = 5cm,
    width = 5cm,
    xmin = -2,
    xmax = 2,
    ymin = -2,
    ymax = 2,
     ]
    \nextgroupplot
    \addplot [only marks, mark size=1] table {
-1.1 1.1
};
    \draw (axis cs:1,1) -- (axis cs:-1,1) -- (axis cs:-1,-1) -- (axis cs:1,-1)
    -- cycle;
    \nextgroupplot
    \addplot [only marks, mark size=1] table {
-1.1 1.1
};
    \draw (axis cs:0,{sqrt(2)}) -- (axis cs:{sqrt(2)},0) -- (axis cs:0,-{sqrt(2)}) -- (axis cs:-{sqrt(2)},0)
    -- cycle;
\nextgroupplot
    \addplot [only marks, mark size=1] table {
-1.1 1.1
};
    \addplot[data cs=polar,domain=0:360,samples=360,smooth]
    (x,{2/sqrt(pi)});
    \nextgroupplot
    \addplot [only marks, mark size=1] table {
-1.1 1.1
};
    \addplot[data cs=polar,domain=0:360,samples=360,smooth]
    (x,{sqrt(4/pi+0.25)});
    \addplot [only marks, mark size=1] table {
-1.1 1.1
};
    \addplot[data cs=polar,domain=0:360,samples=360,smooth]
    (x,0.5);
    \nextgroupplot
    \addplot [only marks, mark size=1] table {
-1.1 1.1
};
    \addplot[data cs=polar,domain=0:45,samples=100]
    (x,{1.449^2/(abs(cos(x-22.5))+abs(sin(x-22.5))+2*sqrt(abs(sin(2*x-45))/2))});
    \addplot[data cs=polar,domain=90:135,samples=100]
    (x,{1.449^2/(abs(cos(x-22.5))+abs(sin(x-22.5))+2*sqrt(abs(sin(2*x-45))/2))});
    \addplot[data cs=polar,domain=180:225,samples=100]
    (x,{1.449^2/(abs(cos(x-22.5))+abs(sin(x-22.5))+2*sqrt(abs(sin(2*x-45))/2))});
    \addplot[data cs=polar,domain=270:315,samples=100]
    (x,{1.449^2/(abs(cos(x-22.5))+abs(sin(x-22.5))+2*sqrt(abs(sin(2*x-45))/2))});
    \addplot[data cs=polar,domain=45:90,samples=100]
    (x,{1.449^2/(abs(cos(x+22.5))+abs(sin(x+22.5))+2*sqrt(abs(sin(2*x+45))/2))});
    \addplot[data cs=polar,domain=135:180,samples=100]
    (x,{1.449^2/(abs(cos(x+22.5))+abs(sin(x+22.5))+2*sqrt(abs(sin(2*x+45))/2))});
    \addplot[data cs=polar,domain=225:270,samples=100]
    (x,{1.449^2/(abs(cos(x+22.5))+abs(sin(x+22.5))+2*sqrt(abs(sin(2*x+45))/2))});
    \addplot[data cs=polar,domain=315:360,samples=100]
    (x,{1.449^2/(abs(cos(x+22.5))+abs(sin(x+22.5))+2*sqrt(abs(sin(2*x+45))/2))});
\end{groupplot}
\end{tikzpicture}
  \caption{Test domains used throughout this section. The dot marks the location
  of the singularity in the third test function.}
  \label{fig:domains}
\end{figure}
\begin{remark}
Faster options exist to approximate functions on rectangular regions, including
tensor product 1D Fourier Extensions. The experiments in this section do not exploit this structure.
\end{remark}
\label{sec:results}
\subsection{Complexity}
\Cref{fig:timings} shows execution time for Fourier frame approximations with
increasing degrees of freedom. The approximant is irrelevant here
since complexity of \cref{alg:dpss2} is independent of the right hand side. Our
algorithm computes the equivalent of a truncated SVD and is applied to the
sampled function in a single step. The domain is also
largely irrelevant since the plunge region size is similar for the chosen
domains, see \cref{sec:domaininfluence}. Therefore the timings are shown for just the one example: approximating
\begin{equation*}
  f(x,y)=e^{(x+y)}\cos(20xy)
\end{equation*}
on a disk of area 4. 
\setlength{\figurewidth}{8cm}
\setlength{\figureheight}{6cm}
\begin{figure}[hbtp]
  \centering
  \includegraphics{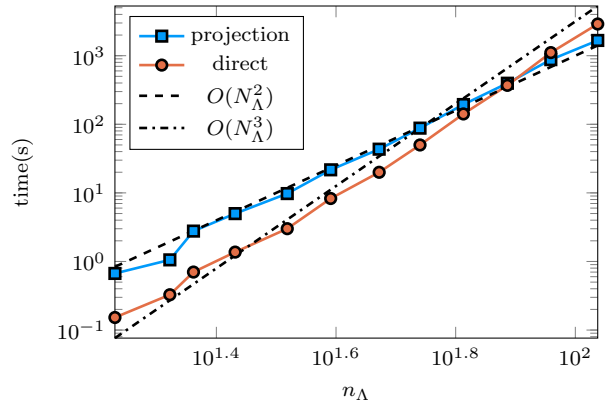}
  \caption{Execution time for a 2D frame approximation, using both a direct
    solver and the projection algorithm. $O(\ndomfp^2)$ and $O(\ndomfp^3)$ shown dashed in black.}
  \label{fig:timings}
\end{figure}

The results confirm the $O(\ndomfp^2\log^2{\ndomfp})$ complexity of the projection
algorithm, with the dominant cost being the SVD used in the randomized SVD solver. Unfortunately, the direct method is only overtaken for $\singlefp>90$, making
the projection method mostly suited for problems requiring a large number of
degrees of freedom, such as oscillatory functions. However, the algorithm
provides control over the regularization through the chosen accuracy $\epsilon$ of the low rank
problem, which the direct solver lacks. In the presence of noise on the order of $\delta$, one may want to choose $\epsilon > \delta$ in order to retain numerical stability \cite{Adcock2016}.

\subsection{Accuracy}
\label{sec:accuracy}
To show convergence, the frame approximation method was applied to a set of test
functions, for increasing degrees of freedom. The test functions are 
\begin{itemize}
\item A well-behaved, smooth function 
  \begin{equation*}
    f(x,y)=e^{x+y}.
  \end{equation*}
\item A function with a singularity inside the bounding box
  \begin{equation*}
    f(x,y)=\frac{1}{((x-1.1)^2+(y-1.1)^2)^2}.
  \end{equation*}
\item An oscillatory function
  \begin{equation*}
    f(x,y)=\cos(24x-32y)\sin(21x-28y).
  \end{equation*}
\item A function with a discontinuity in the first order partial derivatives
  \begin{equation*}
    f(x,y)=|xy|.
  \end{equation*}
\end{itemize}
The results are shown in \cref{fig:raccuracy,fig:accuracy}, for the residual
norm $\Vert Ax-b \Vert_2$ on the one hand and for the largest point error
$\Vert \mathcal{F}-f \Vert_\infty$ on the other hand, sampled randomly in the domain
(10000 samples). There are a few interesting observations to be made regarding
the convergence behavior for different target functions.

\begin{itemize}
\item The approximation error for the smooth function shows superalgebraic
convergence on all domains, strengthening claims in this regard
\cite{Adcock2016}. The only exception is the star-shaped point error.
\item The approximation error for the oscillatory function behaves exactly as
expected, decreasing rapidly once the highest oscillatory mode can be resolved
by the basis functions. 
\item \Cref{fig:accsingular} shows the results for a function with a singularity
right outside the domain of interest. Similar to the 1D case, this results in a
slower, yet still superalgebraic rate of convergence.
\item A function that has $s$ continuous derivatives will exhibit order $s+1$
  convergence, as seen in \cref{fig:accabs}, for the residual error. The
largest point error shows very little, if any, convergence.. 
\end{itemize}
\setlength{\figurewidth}{8cm}
\setlength{\figureheight}{6cm}
\begin{figure}[hbtp]
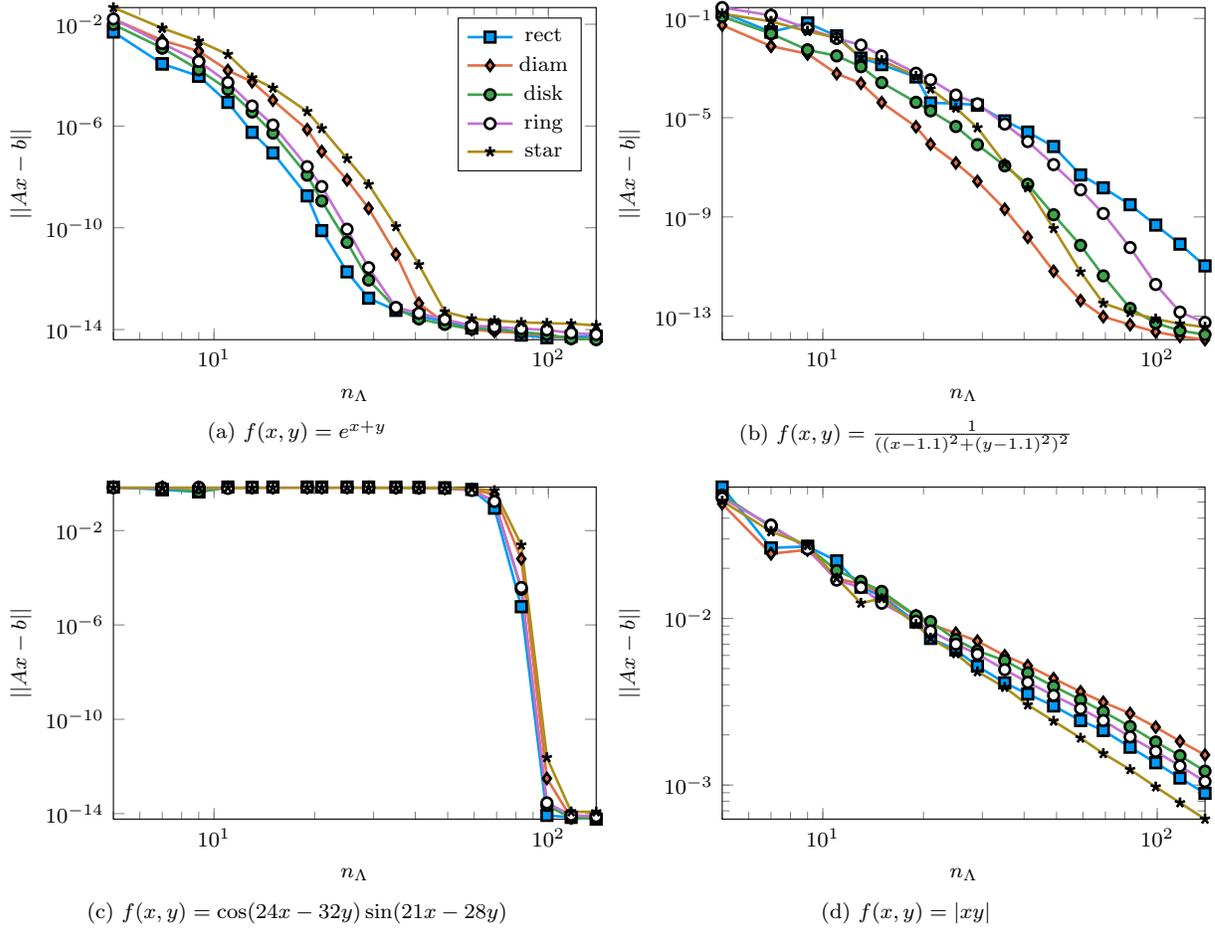

  \centering
\subfloat[$f(x,y)=\e^{x+y}$]{  \includegraphics{img/reasy.tikz}\label{fig:reasy}}
\subfloat[$f(x,y)=\frac{1}{((x-1.1)^2+(y-1.1)^2)^2}$]{  \includegraphics{img/rsingular.tikz}\label{fig:rsingular}}\\
\subfloat[$f(x,y)=\cos(24x-32y)\sin(21x-28y)$]{  \includegraphics{img/roscillatory.tikz}\label{fig:roscillatory}}
\subfloat[$f(x,y)=|xy|$]{  \includegraphics{img/rabs.tikz}\label{fig:rabs}}
  \caption{Residuals for the approximations from \cref{fig:accuracy}.}
  \label{fig:raccuracy}
\end{figure}

\setlength{\figurewidth}{8cm}
\setlength{\figureheight}{6cm}
\begin{figure}[hbtp]
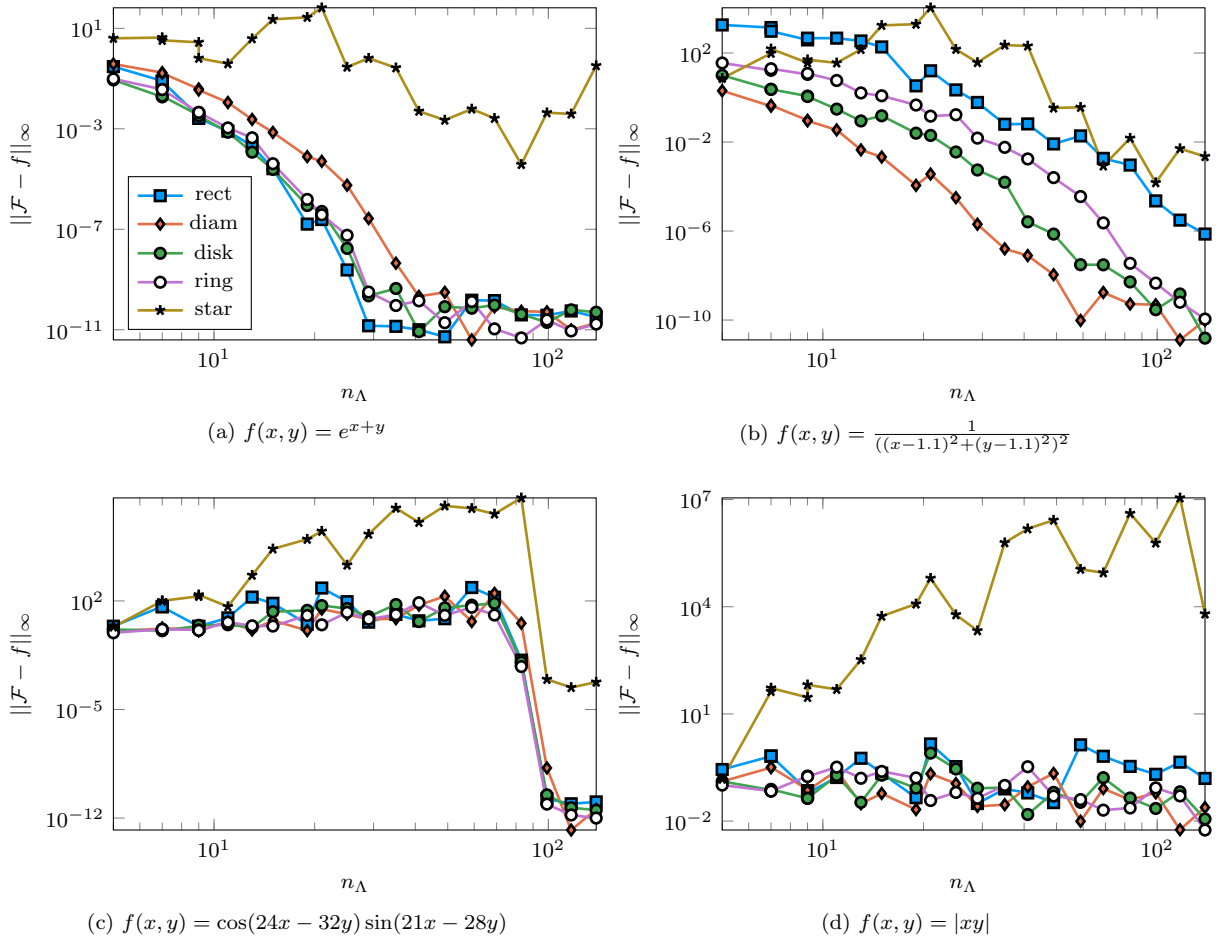

  \centering
\subfloat[$f(x,y)=e^{x+y}$]{  \includegraphics{img/acceasy.tikz}\label{fig:acceasy}}
\subfloat[$f(x,y)=\frac{1}{((x-1.1)^2+(y-1.1)^2)^2}$]{  \includegraphics{img/accsingular.tikz}\label{fig:accsingular}}\\
\subfloat[$f(x,y)=\cos(24x-32y)\sin(21x-28y)$]{  \includegraphics{img/accoscillatory.tikz}}
\subfloat[$f(x,y)=|xy|$]{  \includegraphics{img/accabs.tikz}\label{fig:accabs}}
  \caption{Maximum pointwise error for a 2D frame approximation, for different
    domains and approximants.}
  \label{fig:accuracy}
\end{figure}
\begin{remark}
Our algorithm only ensures small $\Vert Ax-b \Vert$ if a suitable $x$ exists. The effect
of sampling on $\Vert \mathcal{F}-f\Vert$ will be studied in a follow-up paper on \cite{Adcock2016} focusing
on the discrete case, and commented on in the next section.
\end{remark}

\subsection{Influence of domain shape}
\label{sec:domaininfluence}
\subsubsection{Plunge region estimates}
\label{sec:plungeestimate}
\Cref{thm:combined} leads to an estimate of the plunge region of the form
\begin{equation}
  \label{eq:31}
  \Neps=C_1\ndombp\log{\singletp}+O(\ndombp)
\end{equation}
where $C_1=\left(\frac{\singlefp}{\singletp} \right)^{D-1}\frac{1}{\epsilon^2}$. This is because the eigenvalues $\lambda_i$ of $\TBT$ are
the squares of the singular values $\sigma_i$ of the collocation matrix $A$, so
that 
\begin{equation}
  \label{eq:32}
  \epsilon<\sigma_i<1-\epsilon \Leftrightarrow \epsilon^2<\lambda_i<1-\epsilon+\epsilon^2.
\end{equation}
The constant $C_1$ is a gross overestimate, as shown in
\cref{fig:plungeconstant}, which plots the
ratio $\Neps/(\sqrt{\singletp}\log{\singletp})$ as a function of $\ndomfp$.
\setlength{\figurewidth}{8cm}
\setlength{\figureheight}{6cm}
\begin{figure}[hbtp]
  \centering
  \includegraphics{img/plungeconstantb.tikz}
  \caption{Estimate of plunge region size with respect to $\ndombp\log{\singletp}$.}
  \label{fig:plungeconstant}
\end{figure}
The ellipse, square and diamond seem to reach the asymptotic behavior
\cref{eq:31} with a constant $C_1\sim 10$. The ring and star domain have not yet
reached their plateau, but it is reasonable to assume this plateau, like the
bound from \cref{thm:combined}, is proportional in some way to the domain
boundary length. Using the Euclidean length, the plateau for the ring would be
at $10^{1.08}$ and for the star at $10^{1.39}$, both plausible from \cref{fig:plungeconstant}.
\begin{remark}
When using an $\epsilon$ close to machine precision, an alternative to using an estimate for the plunge region rank is to use an
adaptive form of the random matrix algorithm for unknown ranks. As per
\cite{Liberty2007}, this eliminates the need for a difficult estimation of $C_1$, at a maximum factor 2
increase in cost.
\end{remark}
\subsubsection{Influence on convergence}
\label{sec:convinfluence}
The influence of domain shape on convergence is readily apparent from
\crefrange{fig:acceasy}{fig:accabs}. There are a number of factors that combined
lead to the differences seen between the domains.

\paragraph{The maximum pointwise error}
In \cref{fig:accuracy} the error was taken as the infinity norm over $\Omega$
for $F-f$, calculated over 10000 random samples of $\Omega$. However, the actual approximation $F-f$ in all these
experiments was computed from an equispaced grid of collocation points. 
Some points in $\Omega$, e.\ g.\ at the spikes of the star shape may be
far away from the equispaced grid. Since no information about these points was
taken into account, convergence in these areas cannot be expected until they are
sufficiently covered by the grid. This is most apparent in the star-shaped
domain, as it has sharp features that are difficult to cover with an equispaced
grid. This problem is unique to the higher-dimensional case, as in the
one-dimensional problem the endpoints can be guaranteed to be included.

\Cref{fig:raccuracy} contains the experiments from \cref{fig:accuracy} but now
shows the relative error in the grid points only. The results show that all
approximations do converge as expected in the collocation points. The difference
is, as was expected most apparent in the star-shaped region. 

Moreover, this shows that the solution the algorithm provides to the least
squares system (\ref{eq:lsq}) is almost always
accurate up to the supplied tolerance.

\paragraph{Proximity to the singularity}
In the 1D case, the effect of the presence of a singularity on the convergence
rate was first quantified in \cite{Huybrechs2010} and later expanded upon in
\cite{Adcock2012}. They found that the first regime of convergence is
$\Vert \mathcal{F}-f\Vert =c\rho^{-N}$ where $\rho$ depends on the minimum of the
closest distance to the bounding box and the distance to the singularity.

This behaviour returns in \cref{fig:accsingular,fig:rsingular}.
With the singularity of the test function
\begin{equation*}
f(x,y)=\frac{1}{((x-1.1)^2+(y-1.1)^2)^2}
\end{equation*}
located at $(1.1,1.1)$, the rectangle is closest to the singularity. In rough order of
proximity to the singularity, the other
domains are star, ring, circle and diamond. The effect is most
apparent in \cref{fig:rsingular}, where the convergence rate of the error for
the diamond shape is significantly higher than the rate for the rectangle, where
for test functions without singularities they differ much less.
\subsection{Robustness}
To ensure the results remain stable for large $\ndomfp$, \cref{fig:robustness}
shows the approximation of a function 
\begin{equation*}
  f(x,y)=\sin\left(\frac{\singlefp}{2}(x+y)\right)
\end{equation*}
for increasing degrees of freedom $\ndomfp$ on a unit circle. This showcases the
close relationship between the number of degrees of freedom needed per
wavelength and the size of the extension region. For $T=1.2$ and $T=2$, the
extension region is narrow enough for the approximation to resolve the
oscillation. The main difference here is the convergence rate, which is slower
for smaller $T$ as per the 1D case. For $T=3$, the highest frequency mode present
in the Fourier basis of degree $\ndomfp$ is not enough to resolve the function, and it is impossible for convergence to occur.

\setlength{\figurewidth}{8cm}
\setlength{\figureheight}{6cm}
\begin{figure}[hbtp]
  \centering
  \includegraphics{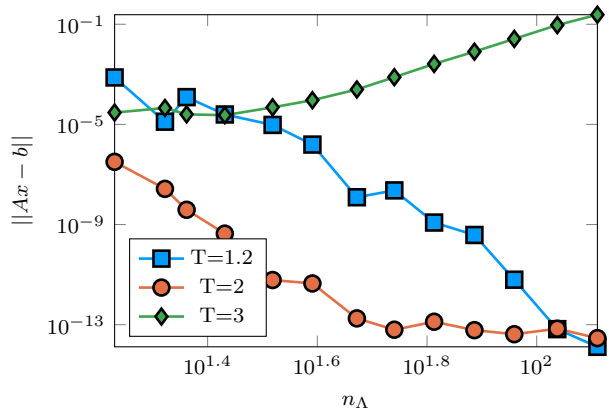}
  \caption{Accuracy for a 2D frame approximation for an increasingly oscillatory
  function, and different extension regions $\bboxt$.}
  \label{fig:robustness}
\end{figure}


\section{Conclusion}
In this paper we examined an algorithm that provides for Fourier Frame
approximations on arbitrary domains. The generality of this approach was
demonstrated by approximating a variety of function on several different
domains, up to machine precision. The ease of implementing and applying the
method, combined with the demonstrated superalgebraic convergence make for an
appealing direction in higher dimensional frame approximations.

Speeding up the algorithm hinges on the occurence of a certain singular value
profile. We proved that when using a collocation approach, the matrices exhibit
this sort of profile. This newfound result connects the complexity of the
algorithm to a certain measure of the domain boundary. This is related to a long history of multidimensional
generalizations of Slepian's Prolate Spheroidal Wave functions, albeit in a
discrete sense.

The natural decoupling of the problem into a lower-dimensional problem
corresponding to the boundary of the domain and a well-conditioned problem in
the interior may very well lead to further improvements in complexity.

\section*{Acknowledgements}

The authors gratefully acknowledge interesting discussions on the topic of this paper with Ben Adcock, Vincent Copp{\'e}, Evren Yarman and Marcus Webb. The authors are supported by FWO Flanders projects G.A004.14 and G.0641.11.

\bibliography{mrabbrev,Paper2D}
\bibliographystyle{abbrv}

\end{document}